\documentclass[10pt]{amsart}
\usepackage{amssymb,amsbsy,amsmath,amsfonts,times, amscd}
\usepackage{latexsym,euscript,exscale}

\usepackage{helvet}

\title{Large scale geometry of metrisable groups}

\author {Christian Rosendal}
\address{Department of Mathematics, Statistics, and Computer Science (M/C 249)\\University of Illinois at Chicago\\851 S. Morgan St.\\Chicago, IL 60607-7045\\USA}
\email{rosendal.math@gmail.com}
\urladdr{http://homepages.math.uic.edu/$~$rosendal}

\date {}
\newcommand{\norm}[1]{\lVert#1\rVert}
\newcommand{\Norm}[1]{\big\lVert#1\big\rVert}
\newcommand{\NORM}[1]{\Big\lVert#1\Big\rVert}
\newcommand{\triple}[1]{|\!|\!|#1|\!|\!|}

\newcommand {\F}{\mathbb F}
\newcommand {\N}{\mathbb N}
\newcommand {\M}{\mathbb M}
\newcommand {\Q}{\mathbb Q}
\newcommand {\R}{\mathbb R}
\newcommand {\Z}{\mathbb Z}

\newcommand {\U}{\mathbb U}

\newcommand{\normal}{\trianglelefteq}

\newcommand{\eps}{\epsilon}

\newcommand{\tom} {\emptyset}

\newcommand{\saa}{\Rightarrow}

\newcommand{\til}{\rightarrow}

\newcommand{\Lim}[1]{\mathop{\longrightarrow}\limits_{#1}}

\newcommand {\Del}{ \; \Big| \;}
\newcommand {\del}{ \; \big| \;}

\newcommand {\ku} {\mathcal}

\newcommand{\ov}{\overline}
\newcommand{\inv}{^{-1}}

\newcommand {\e} {\exists}

\newtheorem{thm}{Theorem}
\newtheorem{cor}[thm]{Corollary}
\newtheorem{lemme}[thm]{Lemma}
\newtheorem{prop} [thm] {Proposition}
\newtheorem{defi} [thm] {Definition}
\newtheorem{obs}[thm] {Observation}

\newtheorem{prob}[thm]{Problem}

\theoremstyle{definition}

\newtheorem{exa}[thm]{Example}

\usepackage[sc]{mathpazo}

\begin{document}

\thanks{This work was partially supported
by a grant from the Simons Foundation (\#229959 to Christian Rosendal) and likewise by  NSF grant DMS 1201295}

\begin{abstract}
We develop a theory of large scale geometry of metrisable topological groups that, in a significant number of cases, allows one to define and identify a unique quasi-isometry type intrinsic to the topological group. 
Moreover, this quasi-isometry type coincides with the classical notion in the case of compactly generated locally compact groups and, for the additive group of a Banach space, is simply that of the corresponding Banach space. In particular, we characterise the class of separable metrisable groups admitting {\em metrically proper}, respectively, {\em maximal} compatible left-invariant  metrics. Moreover, we develop criteria for when a metrisable group admits  metrically proper affine isometric actions on Banach spaces of various degress of convexity and reflexivity.

A further study of the large scale geometry of automorphism groups of countable first order model theoretical structures is separated into a companion paper.
\end{abstract}

\maketitle

\tableofcontents

\section{Introduction}
The large scale geometry of finitely generated discrete groups or compactly generated locally compact second countable groups is by now a well-established theory (see \cite{nowak, harpe} for recent accounts). In the finitely generated case, the starting point is the elementary observation that the word metrics $\rho_\Sigma$ on a discrete group $\Gamma$ given by finite symmetric generating sets $\Sigma\subseteq \Gamma$ are mutually quasi-isometric and thus any such metric may be said to define the large scale geometry of $\Gamma$. In the locally compact setting, matters have not progressed equally swiftly even though the basic tools have been available for quite some time. Indeed, by a result of R. Struble \cite{struble} dating back to 1951, every locally compact second countable group admits a compatible left-invariant {\em proper} metric, i.e., so that the closed balls are compact. Struble's theorem was based on an earlier well-known result due independently to G. Birkhoff \cite{birkhoff} and S. Kakutani \cite{kakutani} characterising the metrisable topological groups as the first countable topological groups and, moreover, stating that every such group admits a compatible left-invariant metric. However, as is evident from the construction underlying the Birkhoff--Kakutani theorem, if one begins with a compact symmetric generating set $\Sigma$ for a locally compact second countable group $G$, then one may obtain a compatible left-invariant metric $d$ that is quasi-isometric to the word metric $\rho_\Sigma$ induced by $\Sigma$. By applying the Baire category theorem and arguing as in the discrete case, one sees that any two such word-metrics $\rho_{\Sigma_1}$ and $\rho_{\Sigma_2}$ are quasi-isometric, which shows that the compatible left-invariant metric $d$ is uniquely defined up to quasi-isometry by this procedure.

However, thus far, it seems that no one has been able to identify a well-defined large scale geometry of metrisable topological groups beyond the locally compact. Largely, this may be  due to the presumed absence of canonical generating sets in general metrisable groups as opposed to  the finitely or compactly generated ones.

In the present paper, we offer a solution to this problem, which in many cases allows one to isolate and compute a canonically defined metric on a metrisable topological group $G$ and thus to identify a unique quasi-isometry type of $G$. Moreover, this quasi-isometry type satisfies the main characterics encountered in the finitely or compactly generated settings, namely (i) that it is a topological isomorphism invariant of $G$, (ii) that it is realised on $G$ by some compatible left-invariant metric and (iii) it is non-trivial, witnessed here by capturing all possible large scale behaviour of $G$.

The theory of metrisable topological groups,  in particular, Polish groups,  has seen tremendous progress in the last twenty years. For example, in the special case of non-Archimedean Polish groups, studies have been done of their topological dynamics and of questions of amenability \cite{kpt}, the structure of conjugacy classes and topological rigidity \cite{turbulence} and their representation theory \cite{tsankov}. Our goal here is to add another facet to the theory, by providing the framework for the introduction of more geometrical tools, connecting this with the structure of affine isometric actions on Banach spaces and also by studying in detail the special case of non-Archimedean Polish groups.

The basic idea of our paper is to replace the topological notion of compactness by a metric notion, which we term the {\em relative property (OB)}. Namely, a subset $A$ of a metrisable topological group $G$ is said to have {\em property (OB) relative to $G$} if $A$ has finite diameter with respect to every left-invariant metric on $G$. Similarly, $G$ is said to have {\em property (OB)} \cite{OB} if it has property (OB) relative to itself, i.e., if $G$ has finite diameter with respect to every compatible left-invariant metric. Whereas, by the theorem of Struble, the relative Property (OB) is equivalent to relative compactness in the case of locally compact second countable groups, this is far from being the case for general metrisable groups. Indeed, a surprising number of these have property (OB) without being compact, e.g., the unitary group $U(\ku H)$ with the strong operator topology \cite{ricard}, homeomorphism groups of spheres ${\rm Homeo}(S^n)$  and of the Hilbert cube ${\rm Homeo}([0,1]^\N)$ \cite{OB}, the group of measure-preserving automorphisms of a standard probability space ${\rm Aut}([0,1]^\N)$ with the weak topology \cite{glasner}, and, more generally, automorphism groups of separably-categorical metric structures \cite{OB}.

We then define a compatible left-invariant metric $d$ on $G$ to be {\em metrically proper} if every set of finite $d$-diameter set has property (OB) relative to $G$ and show that any two such metrics on $G$ will be coarsely equivalent. Moreover,  the existence of these metrics is characterised by the following theorem.
\begin{thm}
A separable metrisable topological group $G$ admits a metrically proper compatible left-invariant metric if and only if $G$ has the {\em local property (OB)}, i.e., there is a neighbourhood of the identity having property (OB) relative to $G$.
\end{thm}
Though many familiar groups do have the local property (OB), there are coun\-terexamples such as the infinite direct product of groups without property (OB), e.g., $\Z^\N$.

While the metrically proper metrics, when they exist,  uniquely define the coarse equivalence class of a metrisable group $G$, they do not suffice to define its quasi-isometry class. For that, we isolate a more restrictive class of metrics. Note first that we may define an ordering on the class of compatible left-invariant metrics on $G$ by letting $\partial\lesssim d$  if there is a constant $K$ so that $\partial\leqslant K\cdot d+K$. We then define a metric $d$ to be {\em maximal} if it maximal with respect to this ordering. Note that any two maximal metrics are necessarily quasi-isometric. Moreover, they can be characterised as follows.
\begin{prop}
The following conditions are equivalent for  a metrically proper compatible left-invariant metric $d$ on a metrisable group $G$,
\begin{enumerate}
\item $d$ is maximal,
\item $(G,d)$ is large scale geodesic,
\item there is $\alpha>0$ so that $\Sigma=\{g\in G\del d(g,1)\leqslant \alpha\}$ generates $G$ and $d$ is quasi-isometric to the word metric $\rho_\Sigma$.
\end{enumerate}
\end{prop}
Similarly to the locally compact case, their existence is characterised by $G$ having an appropriate generating set.
\begin{thm}
A separable metrisable group $G$ admits  a maximal compatible left-invariant metric $d$ if and only if 
$G$ is generated by an open set with property (OB) relative to $G$.
\end{thm}
The {\em quasi-isometry type} of a metrisable group $G$, when it exists, is then that given by a maximal compatible left-invariant metric. A reassuring fact about this definition is that it is a conservative extension of the existing theory. Namely, as the relative property (OB) in a locally compact group coincides with relative compactness, one sees that our definition of the quasi-isometry type of a compactly generated locally compact group coincides with the classical one given in terms of word metrics for compact generating sets. But, moreover, as will be shown, if $(X,\norm\cdot)$ is a Banach space, then the norm-metric will be maximal on the underlying additive group $(X,+)$, whereby $(X,+)$ will have a well-defined quasi-isometry type, namely, that of $(X,\norm\cdot)$. Thus, nothing essentially new will be said on these two classes of objects.

With the aid of the Milnor--\v{S}varc theorem \cite{milnor, svarc} adapted to our setting, we are then capable of computing the quasi-isometry type of a number of different groups. A part from the groups with property (OB) that will be quasi-isometric to the one-point metric space, and the compactly generated locally compact groups and Banach spaces discussed above, we mention the following examples (some of which are studied in the companion paper \cite{large scale}).
\begin{itemize}
\item The groups of affine isometries of $X=L^p$ and $X=\ell^p$, $1<p<\infty$, equipped with the topology of pointwise convergence are quasi-isometric to the Banach space $X$.
\item The automorphism group ${\rm Aut}(\bf T)$ of the $\aleph_0$-regular tree $\bf T$ is quasi-iso\-metric to $T$. Similarly for the group ${\rm Aut}({\bf T},\mathfrak e)$ of automorphisms of $\bf T$ fixing a given  end $\mathfrak e$.
\item The isometry group ${\rm Isom}(\U)$ of the Urysohn metric space $\U$ is quasi-isometric to $\U$.
\end{itemize}
Since the Banach spaces $L^p$ and $\ell^p$, $1<p<\infty$, are known to all belong to distinct quasi-isometry classes, the same holds for their affine isometry groups, whereby, in particular, these latter cannot be isomorphic as topological groups.

The second part of the paper concerns the interplay between the large scale geometry of a group $G$ and its affine isometric actions on Banach spaces. Possibly the seminal result in this area is due to U. Haagerup \cite{haagerup}, who showed that the free group on finitely many generators admits a proper affine isometric action on Hilbert space, but the subject has since broadened to include less regular types of Banach spaces \cite{furman}. One motivation for this study is that embeddability into Banach spaces provide an eminent domain for the calibration of various geometric properties of metric spaces, e.g., in terms of the convexity properties of the Banach space. 

Our first result uses  a well-known construction due to R. F. Arens and J. Eells \cite{arens} to show that every metrisable group $G$ carrying a metrically proper compatible left-invariant metric admits a proper affine isometric action on some Banach space. We subsequently generalise a construction due to I. Aharoni, B. Maurey and B. S. Mityagin \cite{maurey}  to show that amenable metrisable groups admit metrically proper affine isometric actions on Hilbert space if and only if they have uniformly continuous coarse embeddings into Hilbert space.

Assuming a stronger version of amenability, namely, that the group $G$ is approximately compact, we use techniques originating in work of V. Pestov \cite{pestov} to show that $G$ admit metrically proper affine isometric actions on super-reflexive spaces, respectively on  spaces  with Rademacher type $p$ and cotype $q$, if and only if $G$ has a uniformly continuous coarse embedding into a space of the same kind.

Our final result highlights the difference with the theory for locally compact groups. Namely, N. Brown and E. Guentner \cite{BG} have shown that every finitely generated group admits a proper affine isometric action on a reflexive Banach space and this was generalised in \cite{haagerup-affine} to compactly generated locally compact groups. However, the generalisation to metrisable groups fails and thus reflexivity becomes a non-trivial restriction. Though we have not been able to determine exactly which metrisable groups act properly on reflexive spaces, a sufficient criterium is given by the following result.

\begin{thm}
Suppose $d$ is a metrically proper metric on $G$ and that, for all $\alpha>0$, there is a continuous weakly almost periodic function $\phi\in \ell^\infty(G)$  with $d$-bounded support so that $\phi\equiv 1$ on the ball $D_\alpha=\{g\in G\del d(g,1)\leqslant \alpha\}$.
Then $G$ admits a metrically proper continuous affine  isometric action on a reflexive Banach space.
\end{thm}

The hypotheses of this theorem are verified whenever the metric $d$ is moreover {\em stable} in the sense of J.-L. Krivine and B. Maurey \cite{KM}, whence the following corollary.
\begin{cor}
Suppose a topological group $G$ carries a compatible left-invariant metrically proper stable metric. Then $G$ admits a metrically proper continuous affine isometric action on a reflexive Banach space.
\end{cor}
For a particular instance of this,  let us just mention a result of  \cite{large scale} stating that if $\bf M$ is a countable atomic model of a stable first-order theory $T$, whose automorphism group ${\rm Aut}(\bf M)$ has the local property (OB), then ${\rm Aut}(\bf M)$ admits a metrically proper stable metric.

The third and last part of our study is separated into a companion paper \cite{large scale}, which treats the subcase of non-Archimedean Polish groups, i.e., automorphism groups of countable first-order structures.  

\

\noindent {\bf Acknowledgement.} 
I wish to thank P. de la Harpe, W. B. Johnson and A. Pillay for helpful comments on and answers to the subject presented here.


\section{Basic theory}
\subsection{The relative property (OB)}
Let us begin by recalling that a {\em metrisable topological group} is a topological group $G$ whose topology can be induced by some metric. Thus, the metric is not part of the given. Also, while separability is occasionally an issue in our study, this is not so for {\em complete metrisability}, i.e., that the topology on $G$ can be given by a complete metric. In fact, for our purposes, $G$ can be considered to be equivalent with all of its dense subgroups, as long as these are given the topology induced from $G$.

As mentioned in the introduction, the following definition is central to the rest of the paper. It is a rather trivial generalisation of the concept of groups with property (OB), which was initially considered in \cite{OB} and also independently by Y. de Cornulier.

\begin{defi}
Let $A$ be a subset of a metrisable group $G$. We say that $A$ has {\em property (OB) relative to $G$} if $A$ has finite diameter with respect to every compatible left-invariant metric on $G$.
\end{defi}
Here (OB) is an abbreviation of {\em orbites born\'ees}, i.e., bounded orbits, which refers to the fact that, if $G\curvearrowright (X,d)$ is a continuous isometric action of $G$ on a metric space $(X,d)$ and $A$ has property (OB) relative to $G$, then $A\cdot x$ has finite diameter in $X$ for all $x\in X$. This follows from the fact that, if $\partial$ denotes a compatible left-invariant metric on $G$ (which exists by the Birkhoff--Kakutani metrisation theorem discussed below), then $D(g,f)=d(gx,fx)+\partial(g,f)$ is also compatible and left-invariant.

The following observation will be used repeatedly in the following.
\begin{lemme}\label{triangle ineq}
Let $d$ be a left-invariant metric on a group $G$ and let $A,B\subseteq G$ be subsets of finite $d$-diameter. Then also $A\cdot B=\{ab\del a\in A\;\&\; b\in B\}$ and $A\inv=\{a\inv \del a\in A\}$ have finite $d$-diameter.
\end{lemme}

\begin{proof}
Just note that for $a\in A$ and $b\in B$,
$$
d(1,ab)\leqslant d(1,a)+d(a,ab)=d(1,a)+d(1,b)\leqslant \sup_{g\in A}d(1,g)+\sup_{f\in B}d(1,f)
$$
so the products $ab$ have a bounded distance to $1$ and $A\cdot B$ is $d$-bounded. Similarly, 
$d(a\inv,1)=d(a a\inv,a)=d(1,a)\leqslant \sup_{g\in A}d(1,g)$.
\end{proof}

It follows immediately from the Lemma \ref{triangle ineq} that the class of subsets having property (OB) relative to $G$ is an ideal, i.e., is hereditary and closed under taking finite unions, and, moreover, is closed under taking products and inverses of sets.

We remark also that, if $d$ is a compatible left-invariant metric on $G$, then $\partial(g,h)=d(g\inv,h\inv)$ defines a compatible right-invariant metric on $G$ and vice versa. Using this observation and Lemma \ref{triangle ineq}, we see that a subset $A$ has finite diameter with respect every compatible {\em left}-invariant metric if and only if $A\inv$ has finite diameter with respect every compatible {\em left}-invariant metric, which again is equivalent to $A$ having finite diameter with respect every compatible {\em right}-invariant metric. It follows that property (OB) relative to $G$ is equivalently defined in terms of boundedness with respect to compatible right-invariant metrics on $G$.

The fundamental result on metrisability of topological groups is the metrisation theorem of G. Birkhoff \cite{birkhoff} and S. Kakutani \cite{ kakutani} stating that a Hausdorff topological group admits a compatible left-invariant metric if and only if it is first countable (henceforth all topological groups considered will be supposed  Hausdorff). In particular, every metrisable topological group admits a compatible left-invariant metric. However, not every metrisable group admits a compatible two-sided invariant metric. The class of those that do was determined by V. Klee \cite{klee} as the so called {\em SIN groups} (for {\em small invariant neighbourhoods}), namely those that furthermore admit a neighbourhood basis at $1$ consisting of conjugacy invariant sets. 

The underlying fundamental fact for Birkhoff's proof of the  metrisation theorem is the following lemma.

\begin{lemme}\label{birkhoff-kakutani} 
Let $G$ be a topological group and $(V_n)_{n\in\Z}$
a neighbourhood basis at the identity consisting of open sets satisfying, for all $n\in \Z$,
\begin{enumerate}
\item $V_n=V_n\inv$,
\item $G=\bigcup_{n\in \Z}V_n$,
\item $V_n^3\subseteq V_{n+1}$.
\end{enumerate}
Define $\delta(g_1,g_2)=\inf\big(2^n\del g_2\inv g_1\in V_n\big)$ and put
$$
d(g_1,g_2)=\inf \Big(\sum_{i=0}^{k-1}\delta(h_i,h_{i+1})\del h_0=g_1,
h_k=g_2\Big).
$$
Then
$$
\delta(g_1,g_2)\leqslant 2\cdot d(g_1,g_2)\leqslant 2\cdot\delta(g_1,g_2)
$$
and $d$ is a compatible left-invariant metric on $G$.
\end{lemme}

Nevertheless, the Birkhoff construction will, in general, be too crude for our purposes, as the balls $B_\eps=\{g\in G\del d(g,1)\leqslant \eps\}$ increase too fast with $\eps\til \infty$ and, similarly, decrease to fast with $\eps\til 0_+$. We shall mention that, at least for small distances, the more elaborate construction of Kakutani gives a tighter control on the behaviour of the metric. 

Reprising this construction, R. Struble \cite{struble} showed that any metrisable (i.e., second countable) locally compact group admits a compatible left-invariant {\em proper} metric, i.e., so that all bounded sets are relatively compact. To do this, one can simply choose the $V_n$ of Lemma \ref{birkhoff-kakutani} to be relatively compact. 

However, it is not possible to combine the results of Klee and Struble to conclude that a locally compact second countable SIN group admits a two-sided invariant proper metric. For a simple counter-example, take a countable discrete group $\Gamma$ with an infinite conjugacy class $\ku C\subseteq \Gamma$. Then, if $d$ is a compatible two-sided invariant metric on $\Gamma$, pick $g\in \ku C$ and note that $d(fgf\inv,1)=d(g,1)$ for all $f\in \Gamma$, i.e., $\ku C$ lies on a sphere around $1$, showing that $d$ cannot be proper.

Finally, S. Kakutani and K. Kodaira \cite{kodaira} showed that if $G$ is locally compact, $\sigma$-compact, then for any sequence $U_n$ of neighbourhoods of the identity there is a compact normal subgroup $K\subseteq \bigcap_nU_n$ so that $G/K$ is metrisable.

From Struble's theorem, the following observation is immediate.
\begin{lemme}
A subset $A$ of a locally compact second countable group $G$ has property (OB) relative to $G$ if and only if $A$ is relatively compact in $G$.
\end{lemme} 

For non-locally compact $G$, the situation is somewhat more interesting.

\begin{lemme}\label{main OB}
Suppose $G$ is a separable metrisable group and $U$ is a symmetric neighbourhood of $1$. Then there is a compatible left-invariant metric $d$ on $G$ so that the following are equivalent for all subsets $A\subseteq G$,
\begin{enumerate}
\item $A$ has finite $d$-diameter,
\item there are a finite subset $F\subseteq G$ and some $k\geqslant 1$ so that $A\subseteq (FU)^k$.
\end{enumerate}
\end{lemme}

\begin{proof}
Let $F_0=\{1\}\subseteq  F_1\subseteq F_2\subseteq F_3\subseteq \ldots \subseteq G$ be an increasing sequence of symmetric finite sets whose union $\bigcup_n F_n$ is dense in $G$.
Define $V_n=(F_nUF_n)^{3^n}$ and note that each $V_n$ is symmetric with 
$$
V_n^3=\big((F_nV_0F_n)^{3^n}\big)^3\subseteq (F_{n+1}V_0F_{n+1})^{3^{n+1}}=V_{n+1}.
$$
Define also a neighbourhood basis
$$
V_0\supseteq V_{-1}\supseteq V_{-2}\supseteq \ldots \ni 1
$$
at $1$ consisting of symmetric open sets with $V_n^3\subseteq V_{n+1}$ for all $n$. Note that since $\bigcup_nF_n$ is dense in $G$ and $V_0=U$ is non-empty open, we have $G=\bigcup_{n\in \Z} V_n$.

Let now $d$ denote the compatible left-invariant metric defined from $(V_n)_{n\in \Z}$ via Lemma \ref{birkhoff-kakutani}. Then, $U$ has finite $d$-diameter, since, by Lemma \ref{birkhoff-kakutani},  $d(g,1)\leqslant \delta(g,1)\leqslant 2^0=1$ for all $g\in U=V_0$. It follows, by Lemma \ref{triangle ineq}, that $(FU)^k$ has finite $d$-diameter for all finite $F\subseteq G$ and $k\geqslant 1$, which shows (2)$\saa$(1).

For (1)$\saa$(2), simply note that if $A\subseteq G$ has finite $d$-diameter, then, by Lemma \ref{birkhoff-kakutani}, we have $A\subseteq V_n=(F_nUF_n)^{3^n}$ for some $n\geqslant 1$. Since $F_n$ is finite, the result follows.
\end{proof}

\begin{lemme}\label{char of rel OB}
The following are equivalent for a subset $A$ of a separable metrisable group $G$,
\begin{enumerate}
\item $A$ has property (OB) relative to $G$, 
\item for every open $V\ni 1$ there are a finite subset $F\subseteq G$ and some $k\geqslant 1$ so that $A\subseteq (FV)^k$.
\end{enumerate}
\end{lemme}

\begin{proof}
(1)$\saa$(2): Suppose $A$ has property (OB) relative to $G$ and that $V\ni1$ is open. Then $A$ has finite diameter with respect to the metric $d$ constructed in Lemma \ref{main OB} from $U=V\cap V\inv$ and thus, by (1)$\saa$(2) of the same lemma, we have that $A\subseteq (FU)^k\subseteq (FV)^k$ for some finite $F\subseteq G$ and $k\geqslant 1$. 

(2)$\saa$(1): Suppose that (2) holds and that $d$ is a continuous left-invariant metric. Then the open ball $V=\{g\in G\del d(g,1)<1\}$ is a neighbourhood of $1$ and so $A\subseteq (FV)^k$ for some finute $F\subseteq G$ and $k\geqslant 1$. By Lemma \ref{triangle ineq}, $A$ must have finite $d$-diameter.
\end{proof}

We remark that Lemmas \ref{main OB} and  \ref{char of rel OB} may fail when $G$ is no longer required to be separable. To see this, let ${\rm Sym}(\N)$ denote the group of all (not necessarily finitely supported) permutations of $\N$ and equip it with the discrete topology. Then $U=\{1\}$ is a symmetric neighbourhood of the identity and thus ${\rm Sym}(\N)\not\subseteq (FU)^k=F^k$ for all finite $F\subseteq {\rm Sym}(\N)$ and $k\geqslant 1$. Nevertheless, it follows from the main result of G. M. Bergman's paper \cite{bergman} that every left-invariant metric on ${\rm Sym}(\N)$ has bounded diameter and hence ${\rm Sym}(\N)$ has property (OB) relative to itself.

In a general metrisable (as opposed to locally compact) group, the class of relatively compact sets may be very small and not shed much light on the large or even small scale geometry of the group. In the course of our study we shall see that, for metric purposes, the class of subsets with property (OB) relative to $G$ largely replaces the relatively compact sets to the extent that the analogy is occasionally almost trivial.

\begin{defi}
Let $(g_n)$ be a sequence in a metrisable group $G$. We say that $(g_n)$ {\em tends to infinity} and write $g_n\til \infty$ if there is a compatible left-invariant metric $d$ on $G$ so that $d(g_n,1)\til \infty$. 
\end{defi}
Thus,  a subset $A\subseteq G$ has property (OB) relative to $G$ if and only if there is no sequence $(g_n) \subseteq A$ so that $g_n\til \infty$.

Also, in a locally compact second countable group $G$, one has $g_n\til \infty$ if and only if $(g_n)$ eventually leaves every compact set or, equivalently, eventually leaves every set with property (OB) relative to $G$. This, of course, corresponds to the standard use of the notation $g_n\til \infty$.

Similarly, in a metrisable group $G$, we note that, if $g_n\til \infty$, then $(g_n)$ has no subsequence with property (OB) relative to $G$ or, equivalently, that $(g_n)$ eventually leaves every set with property (OB) relative to $G$. It is tempting to believe that this is actually a characterisation of $g_n\til \infty$. Though, as noted above,  this is indeed the case in a locally compact second countable group, we shall  see that this is not true in general and, e.g., fails in the product $\Gamma_1\times \Gamma_2\times\ldots$ of infinite discrete groups.

\begin{prop}\label{products}
Let $H=G_1\times G_2\times\ldots$ be the product of  metrisable groups $G_n$.  Then
\begin{enumerate}
\item a subset $A\subseteq H$ has property (OB) relative to $H$ if and only if $A$ is contained in a product $B_1\times B_2\times \ldots$ of subsets $B_i\subseteq G_i$ with property (OB) relative to $G_i$,
\item a sequence $(h_n)$ in $H$ satisfies $h_n\til \infty$ if and only if there is some $k$ so that 
$$
{\rm proj}_{G_1\times \ldots\times G_k}(h_n)\Lim{n\til \infty}\infty
$$
in $G_1\times \ldots\times G_k$.
\end{enumerate}
\end{prop}

\begin{proof}
(2) Let $(h_n)$ be a sequence in $H$. Assume first that $h_n\til \infty$, as witnessed by some compatible left-invariant metric $d$. Since $d$ is compatible with the topology of $H$, we can find a basic neighbourhood 
$$
U=V_1\times \ldots\times V_k\times G_{k+1}\times G_{k+2}\times \ldots
$$  
of the identity in $H$ so that $U\subseteq \{h\in H\del d(h,1)\leqslant 1\}$. In particular, the subgroup  $G_{k+1}\times G_{k+2}\times \ldots$ has finite $d$-diameter. 
Write now $h_n=a_nb_n$, where $a_n\in G_1\times \ldots\times G_k$ and $b_n\in G_{k+1}\times G_{k+2}\times \ldots$. Then 
$$
d(h_n,1)=d(a_nb_n,1)\leqslant d(a_n,1)+d(b_n,1)\leqslant d(a_n,1)+ {\rm diam}_d(G_{k+1}\times G_{k+2}\times \ldots),
$$ 
whereby, as $d(h_n,1)\til \infty$,  we have $d(a_n,1)\til \infty$ and thus also ${\rm proj}_{G_1\times \ldots\times G_k}(h_n)=a_n\til\infty$. 

Conversely, suppose that ${\rm proj}_{G_1\times \ldots\times G_k}(h_n)\Lim{n\til \infty}\infty$ for some $k$, as witnessed by some compatible left-invariant metric $\partial_1$ on  $G_1\times \ldots\times G_k$. Fix also a compatible left-invariant metric $\partial_2$ on $G_{k+1}\times G_{k+2}\times \ldots$ and let $d$ be defined on $H$ by $d(g_1g_2,f_1f_2)=\partial_1(g_1,f_1)+\partial_2(g_2,f_2)$ for $g_1,f_1\in G_1\times \ldots \times G_k$ and $g_2,f_2\in G_{k+1}\times G_{k+2}\times \ldots$. Then $d(h_n,1)\til \infty$, showing that $h_n\til \infty$ in $H$.

(1) Suppose first that $B_i\subseteq G_i$ for all $i$ and assume that $B_1\times B_2\times \ldots$ fails property (OB) relative to $H$. Then there is a sequence $(h_n)$ in $B_1\times B_2\times \ldots$ so that $h_n\til \infty$ and thus also
$$
{\rm proj}_{G_1\times \ldots\times G_k}(h_n)\Lim{n\til \infty}\infty
$$
for some specific $k\geqslant 1$. Fix a compatible left-invariant metric $d$ on $G_1\times \ldots\times G_k$ witnessing this. Writing $h_n=b_{n,1}b_{n,2}\cdots b_{n,k}$ with $b_{n,i}\in B_i$, we have that 
$$
d(h_n,1)\leqslant d(b_{n,1},1)+d(b_{n,2},1)+\ldots+d(b_{n,k},1),
$$
whereby there must be some $m\leqslant k$ and a subsequence $(b_{n_i,m})_i$ satisfying that  $d(b_{n_i,m},1)\Lim{i\til \infty} \infty$. Since $d$ restricts to a compatible left-invariant metric on $G_m$, this shows that $B_m$ fails property (OB) relative to $G_m$.

Conversely, suppose  $A\subseteq H$ has property (OB) relative to $H$. Then there is no sequence $(h_n)$ in $A$ with $h_n\til \infty$, which by (2) implies that, for every $k$,  there can be no sequence $(f_n)$ in ${\rm proj}_{G_k}(A)$ so that $f_n\til \infty$. In other words, $B_k={\rm proj}_{G_k}(A)$ has property (OB) relative to $G_k$ for every $k$ and $A\subseteq B_1\times B_2\times \ldots$. 
\end{proof}     

\begin{prop}
Let $G_1,G_2,\ldots$ be an infinite sequence of metrisable groups without property (OB) and set $H=G_1\times G_2\times \ldots$.   Then there is a sequence $(h_m)$ of elements in $H$ that eventually leaves every set with property (OB) relative to $H$, but so that $h_m\not \til \infty$.
\end{prop}

\begin{proof}
Since each $G_n$ fails property (OB), it contains a sequence $(a_{n,m})_{m\in \N}$ tending to $\infty$. We define
$$
f_{n,k}=(a_{1,k}, a_{2,k},\ldots, a_{k-1,k},a_{k,k}, a_{k+1,n}, a_{k+2,n},\ldots)\in H
$$
and well-order the double sequence $(f_{n,k})_{(n,k)\in \N^2}$ in ordertype $\N$. Now, for every $k$, we have 
$$
{\rm proj}_{G_1\times \ldots\times G_k}(f_{n,m})=(a_{1,k}, a_{2,k},\ldots, a_{k-1,k},a_{k,k})
$$
for an infinite number of $(n,m)\in \N^2$, which, by Proposition \ref{products} (2), shows  that $(f_{n,m})_{(n,m)\in \N^2}$ does not tend to infinity in $H$. 

On the other hand, to see that $(f_{n,m})_{(n,m)\in \N^2}$ eventually leaves every set with property (OB) relative to $H$, we have to see that, whenever $B_i\subseteq G_i$ have property (OB) relative to $G_i$, then only finitely many $f_{n,m}$ belong to $B_1\times B_2\times\ldots$. So fix $k_0$ large enough so that $a_{1,k}\notin B_1$ for $k\geqslant k_0$ and $n_0$ so that $a_{k_0,n}\notin B_{k_0}$ for all $n\geqslant n_0$. 
Suppose that $f_{n,k}\in B_1\times B_2\times \ldots$. As the first coordinate of $f_{n,k}$ is $a_{1,k}$, 
this implies that $k<k_0$ and hence also that the $k_0$th coordinate of $f_{n,k}$ is $a_{k_0,n}$. As we must have $a_{k_0,n}\in B_{k_0}$, we furthermore see that  $n<n_0$. So only finitely many $f_{n,m}$ belong to $B_1\times B_2\times \ldots$. 
\end{proof}

The following definition will be important in the next section.
\begin{defi}
A metrisable group $G$ is said to have the {\em local property (OB)} if there is a neighbourhood $U\ni 1$ with property (OB) relative to $G$.
\end{defi}

By Lemma \ref{products} (1), we immediately see that an infinite product $H=G_1\times G_2\times \ldots$ of metrisable groups without property (OB), e.g., $H=\Z^\N$, cannot have the local property (OB).

Though the relative property (OB) to a large extent is analogous to relative compactness in locally compact groups, there is one major difference which should be kept in mind. Namely, if $G$ is locally compact second countable and $H\leqslant G$ is a locally compact subgroup, then $H$ is closed in $G$ and thus a subset $A\subseteq H$ is relatively compact in $H$ if and only if it is relatively compact in $G$. On the other hand, as shown by V. V. Uspenski\u\i{} \cite{uspenski}, every Polish group is $H$ isomorphic to a closed subgroup  of ${\rm Homeo}([0,1]^\N)$ and by \cite{OB} this latter group has property (OB). Thus, every Polish group $H$ has property (OB) with respect to some ambient Polish group.


\subsection{Metrically proper metrics}Our aim is now to define and investigate the metric equivalent of proper metrics on locally compact groups. However, in contradistinction to locally compact metrisable groups, that by Struble's theorem always admit proper metrics, these may not exist on separable metrisable groups.

\begin{defi}
Let $F\colon (X,d_X) \til (Y,d_Y)$ be a (possibly discontinuous) map between metric spaces $(X,d_X)$ and $(Y,d_Y)$. We define the {\em compression modulus} $\kappa_1\colon [0,\infty[\til [0,\infty]$ of $F$ by
$$
\kappa_1(t)=\inf\big(d_Y(Fx_1, Fx_2)\del d_X(x_1,x_2)\geqslant t\big)
$$
and the {\em expansion modulus} $\kappa_2\colon  [0,\infty[\til [0,\infty]$  by
$$
\kappa_2(t)=\sup\big(d_Y(Fx_1, Fx_2)\del d_X(x_1,x_2)\leqslant t\big).
$$
\end{defi}
Remark that $\kappa_1$ and $\kappa_2$ are non-decreasing functions and that $F$ is uniformly continuous exactly when $\lim_{t\til 0_+}\kappa_2(t)=0$. Also, $F$ is a {\em uniform embedding}, i.e., a uniform homeomorphism with its image, if and only if $F$ is uniformly continuous and $\kappa_1(t)>0$ for all $t>0$.

We can now use our compression and expansion moduli to identify several classes of maps between metric spaces by their behaviour on large distances. Unfortunately, there is not general agreement on the terminology in this area, but only on which classes of maps to be studied.

\begin{defi}
A map $F\colon (X,d_X) \til (Y,d_Y)$ between metric spaces is said to be 
\begin{enumerate}
\item {\em metrically proper} if $F[A]$ has infinite $d_Y$-diameter for every set $A\subseteq X$ of infinite $d_X$-diameter,
\item{\em  expanding} if $\lim_{t\til \infty}\kappa_1(t)=\infty$, 
\item {\em bornologous} if $\kappa_2(t)<\infty$ for all $t\in [0,\infty[$,
\item {\em Lipschitz for large distances} if there is a constant $K\geqslant 1$ so that 
$$
 d_Y(Fx_1, Fx_2)\leqslant K\cdot d_X(x_1,x_2),
$$
whenever $d_X(x_1,x_2)\geqslant K$,
\item {\em Lipschitz for short distances} if there are constants $K\geqslant 1$ and $\eps>0$ so that 
$$
 d_Y(Fx_1, Fx_2)\leqslant K\cdot d_X(x_1,x_2),
$$
whenever $d_X(x_1,x_2)\leqslant \eps$,
\item{\em cobounded} if $\sup_{y\in Y}d_Y(y,F[X])<\infty$,
\item a {\em quasi-isometric embedding} if there is a constant $K\geqslant 1$ so that 
$$
\frac 1K d_X(x_1,x_2)\leqslant d_Y(Fx_1, Fx_2)\leqslant K\cdot d_X(x_1,x_2),
$$
whenever $d_X(x_1,x_2)\geqslant K$,
\item a {\em coarse embedding} if $F$ is both expanding and bornologous,
\item a {\em quasi-isometry between $(X,d_X)$ and $(Y,d_Y)$} if $F$ is a cobounded quasi-isometric embedding,
\item a {\em coarse equivalence between $(X,d_X)$ and $(Y,d_Y)$} if $F$ is a cobounded coarse embedding.
\end{enumerate}
\end{defi}
A few words on equivalent formulations of these concepts are in order. 

(i) Note that $F$ is metrically proper if and only if $F\inv(B)$ has finite diameter for every set $B\subseteq Y$ of finite diameter, which happens if and only if $d_Y(Fx_0, Fx_n)\til \infty$ whenever $(x_n)_{n=0}^\infty$ is a sequence in $X$ satisfying $d_X(x_0,x_n)\til \infty$. 

(ii) Expanding is simply a uniform version of being metrically proper. Equivalently, $F$ is expanding if and only if, for all $R>0$, there is $S>0$ so that
$$
d_X(x_1,x_2)>S\saa d_Y(Fx_1,Fx_2)>R.
$$

(iii) Similarly, $F$ is bornologous if and only if, for all $R>0$, there is $S>0$ so that
$$
d_X(x_1,x_2)<R\saa d_Y(Fx_1,Fx_2)<S.
$$
Thus, a bijection $F$ is bornologous if and only if $F\inv$ is expanding. 

(iv) $F$ is Lipschitz for large distances if and only if $\kappa_2$ is bounded by an affine function, i.e., if there are constants $K,C$ so that
$$
 d_Y(Fx_1, Fx_2)\leqslant K\cdot d_X(x_1,x_2)+C.
 $$

(v) $F$ is Lipschitz for short distances if and only if it is uniformly continuous and its modulus of uniform continuity is bounded below in a neighbourhood of $0$ by a positive linear function.

(vi) $F$ is a quasi-isometry between  $(X,d_X)$ and  $(Y,d_Y)$ if and only if $F$ is a quasi-isometric embedding and there is a quasi-isometric embedding $H\colon (Y,d_Y)\til (X,d_X)$ so that $H\circ F$ is close to ${\rm id}_X$ and $F\circ H$ is close to ${\rm id}_Y$. Here two maps $\phi, \psi\colon Z\til (V,d)$ are {\em close} if $\sup_{z\in Z} d(\phi(z),\psi(z))<\infty$. A similar reformulation is valid for coarse equivalence. 

(vii) Defining  metric spaces $(X,d_X)$ and $(Y,d_Y)$ to be {\em quasi-isometric} or {\em coarsely equivalent} if there is a quasi-isometry, respectively a coarse equivalence, from  $(X,d_X)$ to $(Y,d_Y)$, we see using (vi) that these notions are indeed equivalence relations on the class of metric spaces.

\begin{defi}
A compatible left-invariant metric $d$ on a topological group $G$ is said to be {\em metrically proper} if $d(g_n,1)\til \infty$ whenever $g_n\til \infty$.
\end{defi}

We now have the following characterisation of metrically proper left-invariant metrics.

\begin{lemme}\label{metrically proper}
The following are equivalent for a compatible left-invariant metric $d$ on  a  metrisable group $G$,
\begin{enumerate}
\item $d$ is metrically proper, 
\item for all $A\subseteq G$, we have ${\rm diam}_d(A)<\infty$ if and only if $A$ has property (OB) relative to $G$, 
\item for every compatible left-invariant metric $\partial$ on $G$, the mapping 
$$
{\rm id}\colon (G,\partial)\til (G,d)
$$
is {\em metrically proper},
\item for every compatible left-invariant metric $\partial$ on $G$, the mapping 
$$
{\rm id}\colon (G,d)\til (G,\partial)
$$
is {\em bornologous}.
\end{enumerate}
\end{lemme}

\begin{proof}
(1)$\saa$(2): Suppose $d$ is metrically proper and $A\subseteq G$ is a subset without property (OB) relative to $G$. Then there is some sequence $g_n\in A$ so that $g_n\til \infty$. It follows that also $d(g_n,1)\til \infty$, whereby ${\rm diam}_d(A)=\infty$. 

(2)$\saa$(3): Suppose that (2) holds and $\partial$ is a compatible left-invariant metric on $G$. Then any set $A$ of infinite $\partial$-diameter must fail to have property (OB) relative to $G$ and thus, by (2), must have infinite $d$-diameter.

(3)$\saa$(4):  Assume that (4) fails, i.e.,  that there are a compatible left-invariant metric $\partial$ on $G$, a constant $K>0$ and $g_n, f_n\in G$ so that $d(g_n, f_n)<K$, but that $\partial(g_n,f_n)\til \infty$. 
Then, by left-invariance, we have $d(f_n\inv g_n,1)=d(g_n,f_n)<K$ and $\partial(f_n\inv g_n,1)=\partial(g_n,f_n)\til \infty$, which shows that $\{f_n\inv g_n\}_{n\in \N}$ has infinite $\partial$-diameter, but finite $d$-diameter, which implies the failure of  (3).

(4)$\saa$(1): If $d$ is not metrically proper, there is a sequence $g_n\til \infty$ so that $d(g_n,1)\not\til \infty$. Then, if $\partial$ is a compatible left-invariant metric so that $\partial(g_n,1)\til \infty$, we see that (4) fails for $\partial$. 
\end{proof}

We note that, by condition (4) of Lemma \ref{metrically proper}, if $d$ and $\partial$ are two metrically proper compatible left-invariant metrics on $G$, then the identity map is a coarse equivalence between $(G,d)$ and $(G,\partial)$. It follows that a map into a metric space, $f\colon G\til (X,d_X)$, is bornologous with respect to $d$ if and only if it is bornologous with respect to $\partial$. Similarly with ``metrically proper'' in place of ``bornologous''. Therefore, for groups $G$ admitting such metrics, we can talk of {\em bornologous} and {\em metrically proper} maps without referring to any specific choice of metric on $G$.  

\begin{defi} 
Two compatible left-invariant metrics $d$ and $\partial$ on a group $G$ are said to be {\em coarsely equivalent} or {\em quasi-isometric} if the identity map is a coarse equivalence, respectively a quasi-isometric equivalence, between $(G,d)$ and $(G,\partial)$. 
\end{defi}
Since metrically proper metrics are all coarsely equivalent, we define the {\em coarse equivalence class} of $G$ to be that determined by any metrically proper compatible left-invariant metric, provided such a metric exists.

Though far from being an equivalent reformulation, the following criterion is useful in identifying metrically proper metrics.  
\begin{exa}\label{geodesic proper}
Suppose that $d$ is a compatible  left-invariant {\em geodesic} metric on $G$, i.e., so that any two elements $g,f\in G$ are connected by a continuous path of length $d(g,f)$.  Then $d$ is metrically proper. 

To see this, suppose that $\partial$ is another compatible left-invariant metric on $G$. Since $d$ and $\partial$ are equivalent, there is some $\eps>0$ so that the $d$-ball $B_\eps=\{ g\in G\del d(g,1)\leqslant \eps\}$ has finite $\partial$-diameter. As $d$ is geodesic, we see that, for all $n\geqslant 1$ and $f\in B_{n\eps}$,  there are $g_0=1, g_1, g_2, \ldots, g_{n-1}, g_n=f$ so that $d(1, g_i\inv g_{i+1})=d(g_i,g_{i+1})=\eps$, whence $f= (g_0g_1\inv)\cdot (g_1\inv g_2)\cdots(g_{n-1}\inv g_n)\in \big(B_\eps)^n$.  In other words, $\big(B_\eps)^n=B_{n\eps}$. Since $\big(B_\eps)^n$ has finite $\partial$-diameter for all $n$, it follows that every $d$-bounded set is $\partial$-bounded, which verifies condition (3) of Lemma \ref{metrically proper}.

We mention that, of course, much weaker conditions on $d$ suffice for this argument. Indeed, we only need that, for all $\eps>0$ and $K>0$, there is some $m=m(\eps,K)$ so that every two elements $f,g\in G$ with $d(f,g)\leqslant K$ can be connected by a path $h_0=f, h_1, h_2, \ldots, h_m=g$ where $d(h_i,h_{i+1})\leqslant \eps$ for all $i$.
\end{exa}

\begin{exa}\label{banach1}
Consider the additive group $(X,+)$ of a Banach space $(X,\norm\cdot)$. Since $X$ is clearly geodesic with respect to the metric induced by the norm, by Example \ref{geodesic proper}, we see that this latter is metrically proper on $(X,+)$. 
\end{exa}

It is now time to characterise groups admitting a metrically proper compatible  metric.

\begin{thm}\label{existence of metrically proper}
The following are equivalent for a separable metrisable group $G$, 
\begin{enumerate}
\item $G$ admits a metrically proper compatible left-invariant metric,
\item $G$ has the local property (OB).
\end{enumerate}
\end{thm}

\begin{proof}
(2)$\saa$(1): Suppose $U\ni 1$ is a symmetric open neighbourhood of $1$ with property (OB) relative to $G$ and let $d$ denote the metric on $G$ given by Lemma \ref{main OB}. We claim that $d$ is metrically proper. To see this, we verify condition (2) of Lemma \ref{metrically proper}. So suppose that $A\subseteq G$ has finite $d$-diameter. Then, by Lemma \ref{main OB}, there are a finite set $F\subseteq G$ and  $k\geqslant 1$ so that $A\subseteq (FU)^k$, which, since $U$ has finite diameter with respect to every compatible left-invariant metric, implies that also $A$ has finite diameter with respect to every compatible left-invariant metric. In other words, $A$ has property (OB) relative to $G$, thus verifying condition (2) of Lemma \ref{metrically proper}.

(1)$\saa$(2): Note  that if $d$ is metrically proper, then there is no sequence $g_n$ in $U=\{g\in G\del d(g,1)<1\}$ so that $g_n\til \infty$, i.e., $U$ has property (OB) relative to $G$.
\end{proof}

\begin{exa}
Again, Theorem \ref{existence of metrically proper} fails without the assumption of separability. To see this, consider the free non-abelian group $\F_{\aleph_1}$ on $\aleph_1$ generators $(a_\xi)_{\xi<\aleph_1}$ ($\aleph_1$ is the first uncountable cardinal number). When equipped with the discrete topology, $\F_{\aleph_1}$ is metrisable and has the local property (OB).
On the other hand, if $d$ is any left-invariant metric on $\F_{\aleph_1}$, then some finite-diameter ball $B_k=\{g\in \F_{\aleph_1}\del d(g,1)\leqslant k\}$ must contain a denumerable set of generators $\{a_\xi\}_{\xi \in A}$ ordered as $\N$. Noting that $\F_{\aleph_1}=\F(A)*\F(\aleph_1\setminus A)$, we can then choose some metric $d_1$ on $\F(A)$ for which 
$$
d_1(a_\xi,1)\Lim{\xi \in A}\infty
$$ 
and any other metric $d_2$ on $\F(\aleph_1\setminus A)$ and then define a metric $\partial$ on $\F(\aleph_1)$ by
$$
\partial(g_1f_1\cdots g_nf_n,1)=\sum_{i=1}^n\big(d_1(g_i,1)+d_2(f_i,1)\big)
$$
for $g_i\in \F(A)$ and $f_i\in \F(\aleph_1\setminus A)$. Then $\partial$ witnesses that $a_\xi\Lim{\xi \in A}\infty$ and hence  shows that $d$ is not metrically proper. 

Nonetheless, every  left-invariant metric on the uncountable discrete group ${\rm Sym}(\N)$ is metrically proper, since any such metric is bounded by the main result of \cite{bergman}.
\end{exa}

\begin{defi}
An isometric action $G\curvearrowright (X,d)$ of a metrisable group $G$ on a metric space $(X,d)$ is said to be {\em metrically proper} if, for all $x\in X$, we have
$$
d(g_nx,x)\til \infty 
$$
whenever $g_n\til \infty$ in $G$.
\end{defi}

Setting $x$ to be the identity element $1$ in $G$, one sees that a compatible left-invariant metric $d$ on $G$ is metrically proper if and only if the left-multiplication action $G\curvearrowright (G,d)$ is metrically proper.

\begin{prop}
The following are equivalent for a separable metrisable group $G$,
\begin{enumerate}
\item $G$ admits a metrically proper compatible left-invariant metric,
\item $G$ admits a metrically proper continuous isometric action on a metric space.
\end{enumerate}
\end{prop}

\begin{proof}
(2)$\saa$(1): Supose $G\curvearrowright (X,d)$ is a metrically proper continuous isometric action  on a metric space and let $D$ be any compatible left-invariant metric on $G$. Fix any $x\in X$ and define $\partial(g,f)=d(gx,fx)+D(g,f)$ and note that $\partial$ is left-invariant, continuous and $\partial \geqslant D$, whereby it is compatible with the topology on $G$. Moreover, since the action is metrically proper, so is the metric $\partial$.
\end{proof}


\subsection{Left and right-invariant geometric structures}
As mentioned earlier, even if $G$ has a compatible two-sided invariant metric and  a compatible metrically proper left-invariant metric, it may not have a compatible metric that is simultaneously metrically proper and two-sided invariant. However, those groups that do are easily characterised by the following theorem.

\begin{thm}\label{SIN metrically proper}
The following are equivalent for a separable metrisable group $G$,
\begin{enumerate}
\item $G$ admits a metrically proper two-sided invariant compatible metric $d$,
\item $G$ is a SIN group, has the local property (OB) and, if a subset $A$ has property (OB) relative to $G$, then so does $A^G=\{gag\inv \del a\in A\;\&\; g\in G\}$.
\end{enumerate}
\end{thm}

\begin{proof}
The implication (1)$\saa$(2) is trivial since every $d$-ball is conjugacy invariant in $G$.

For the implication (2)$\saa$(1), note that, since $G$ is a SIN group, we can find a neighbourhood basis $V_0\subseteq V_{-1}\supseteq V_{-2}\supseteq \ldots\ni 1$ at the identity consisting of conjugacy invariant symmetric open sets. Moreover, since $G$ has the local property (OB), some $V_{n}$ has property (OB) relative to $G$. Without loss of generality, we can assume that this happens already for $V_0$. 
Now, by iteration, define 
$$
V_{n+1}=V_n^3\cup \{g\in G\del d(g,1)< n\}^G.
$$
Then each $V_n$ is symmetric open, conjugacy invariant, has property (OB) relative to $G$, $G=\bigcup_nV_n$ and $V_n^3\subseteq V_{n+1}$. It follows that the metric defined from $(V_n)_{n\in\Z}$ via Lemma \ref{birkhoff-kakutani} is two-sided invariant, metrically proper and compatible with the topology on $G$. 
\end{proof}

Even for a non-SIN group $G$, the large scale structure may be two-sided invariant, as, for example, is the case with $S_\infty\times \Z$, which is non-SIN, but, as $S_\infty$ has property (OB), is abelian at the large scale. Again, this admits a charaterisation as in Theorem 
\ref{SIN metrically proper}.

\begin{thm}\label{large scale invariant}
The following are equivalent for a metrisable group $G$ and compatible metrically proper left-invariant metric $d$ on $G$,
\begin{enumerate}
\item for every compatible  right-invariant metric $\partial$ on $G$, the mapping 
$$
{\rm id}\colon (G,d)\til (G,\partial)
$$
is bornologous,
\item  if a subset $A$ has property (OB) relative to $G$, then so does $A^G$.
\end{enumerate}
\end{thm}

\begin{proof}
(1)$\saa$(2): Suppose that (1) holds and that $A$ is a subset with property (OB) relative to $G$. Without loss of generality, we may assume that $1\in A$. Let $\partial(g,h)=d(g\inv, h\inv)$, which is a compatible right-invariant metric, whereby  ${\rm id}\colon (G,d)\til (G,\partial)$ is bornologous. There is therefore some $K>0$ so that 
$\partial(g,h)\leqslant K$, whenever $d(g,h)\leqslant {\rm diam}_d(A)$. It follows that, for all $a\in A$ and $g\in G$, we have 
$$
d(g,ga\inv)=d(1,a\inv)=d(a,1)\leqslant {\rm diam}_d(A),
$$
whereby
$$
d(1, gag\inv)=d(g\inv, ag\inv)=\partial(g, ga\inv)\leqslant K.
$$
It follows that ${\rm diam}_d(A^G)\leqslant 2K$, which, as $d$ is metrically proper, implies that $A^G$ has property (OB) relative to $G$. 

(2)$\saa$(1): Assume that (2) holds and that $\partial$ is a compatible right-invariant metric on $G$. Fix $C>0$, whereby $\{g\in G\del d(g,1)\leqslant C\}^G$ has property (OB) relative to $G$. Now, since the relative property (OB) is equivalently defined in terms of boundedness for right-invariant metrics, there is a $K$ such that 
$$
\{g\in G\del d(g,1)\leqslant C\}^G\subseteq \{h\in G\del \partial(h,1)\leqslant K\}.
$$
 Then, if $d(g,h)\leqslant C$, also $d(h\inv g,1)\leqslant C$ and thus $\partial(g,h)=\partial(gh\inv,1)=\partial\big(h(h\inv g)h\inv, 1\big)\leqslant K$, showing that ${\rm id}\colon (G,d)\til (G,\partial)$ is bornologous. 
\end{proof}

Before stating the next corollary, let us observe that if $d$ is a metrically proper left-invariant metric on $G$, then $\partial$ given by $\partial(g,h)=d(g\inv, h\inv)$ is metrically proper and right-invariant. This can be seen by noting that $d$-balls and $\partial$-balls coincide.
\begin{cor}
Suppose that $d_L$ and $d_R$ compatible metrically proper left-, respectively right-invariant, metrics on $G$. Then ${\rm id}\colon (G,d_L)\til (G,d_R)$ is a coarse equivalence if and only if, whenever a subset $A$ has property (OB) relative to $G$, then so does $A^G$.
\end{cor} 
We designate the above situation by saying that the {\em left and right  geometric structures are coarsely equivalent}.

We should remark that, since the relative property (OB) is equivalently characterised by compatible right-invariant metrics, we have that, if $d$ is metrically proper and left-invariant, then ${\rm id}\colon  (G,\partial)\til (G,d)$ is metrically proper for every compatible right-invariant $\partial$. However, by the above, ${\rm id}$ may fail to be so uniformly, i.e., it may not be expanding.





\subsection{Maximal metrics}
Having studied the metric equivalent of proper metrics on locally compact groups, we now focus on those that are quasi-isometric to the word metrics induced by canonical generating sets.

\begin{defi}
A compatible left-invariant metric $d$ on $G$ is said to be {\em maximal} if, for every other compatible left-invariant metric $\partial$ on $G$, the map
$$
{\rm id}\colon (G,d)\til (G,\partial)
$$
is Lipschitz for large distances,
\end{defi}

We remark that, unless $G$ is discrete, this is really the strongest notion of maximality possible for $d$. Indeed, if $G$ is non-discrete and thus $d$ takes arbitrarily small values, then 
$$
{\rm id}\colon (G,d)\til (G, \sqrt d)
$$
fails to be Lipschitz for small distances, while $\sqrt d$ is a compatible left-invariant metric on $G$.

Also, since a map that is Lipschitz for large distances must be bornologous, we see, by Lemma \ref{metrically proper}, that a maximal metric is always metrically proper. Moreover, any two maximal metrics are clearly bi-Lipschitz equivalent for large distances, i.e., are quasi-isometric.

Recall that, if $\Sigma$ is a symmetric generating set for a group $G$, then we can define an associated {\em word metric} $\rho_\Sigma\colon G\til \N$ by
$$
\rho_\Sigma(g,h)=\min \big(k\geqslant 0\del \e s_1, \ldots, s_k\in \Sigma\; \; g=hs_1\cdots s_k\big).
$$
Thus, $\rho_\Sigma$ is a left-invariant metric on $G$, but, since it only takes values in $\N$, it will never be a compatible metric on $G$ unless of course $G$ is discrete. However, in certain cases, this may be remedied.

\begin{lemme}\label{construction maximal metrics}
Suppose $d$ is a compatible left-invariant metric on a topological group $G$ so that some ball $B_\eps=\{g\in G\del d(g,1)\leqslant \eps\}$ generates $G$. Define $\partial$ by
$$
\partial(f,h)=\inf\Big(\sum_{i=1}^n d(g_i,1)\Del g_i\in B_\eps\; \&\; f=hg_1\cdots g_n\Big).
$$
Then $\partial$ is a compatible left-invariant metric,  quasi-isometric to the word metric $\rho_{B_\eps}$. 
\end{lemme}

\begin{proof}By the continuity of $d$, $\partial$ is a continuous left-invariant {\em \'ecart}, i.e., satisfies the triangle inequality. Also, since $\partial\geqslant d$, it is a metric generating the topology on $G$.

To see that $\partial$ is quasi-isometric to $\rho_{B_\eps}$, note first that 
$$
\partial(f,h)\leqslant \eps\cdot \rho_{B_\eps}(f,h).
$$
For the other direction, fix $f,h\in G$ and  find a shortest sequence $g_1, \ldots, g_n\in B_\eps$ so that $f=hg_1\cdots g_n$
and $\sum_{i=1}^n d(g_i,1)\leqslant \partial(f,h)+1$. Note that, for all $i$, we have  $g_ig_{i+1}\notin B_\eps$, since otherwise we could coalesce $g_i$ and $g_{i+1}$ into a single term $g_ig_{i+1}$ to get a shorter sequence where $d(g_ig_{i+1}, 1)\leqslant d(g_i,1)+d(g_{i+1},1)$. It thus follows that either $g_i\notin B_{\frac \eps2}$ or $g_{i+1}\notin B_\frac\eps2$, whereby there are at least $\frac{n-1}2$ terms $g_i$ so that $g_i\notin B_\frac\eps2$ and hence $d(g_i,1)>\frac\eps2$. In particular, 
$$
\frac{n-1}2\cdot \frac \eps 2< \sum_{i=1}^n d(g_i,1)\leqslant \partial(f,h)+1
$$
and so, as $\rho_{B_\eps} (f,h)\leqslant n$, we have
$$
  \frac \eps4 \cdot \rho_{B_\eps}(f,h)-   \big(1+\frac \eps4\big)\leqslant    \partial(f,h)\leqslant \eps\cdot \rho_{B_\eps}(f,h)
$$
showing that $\partial$ and $\rho_{B_\eps}$ are quasi-isometric.
\end{proof}

Let us also observe that, if $\Sigma$ and $\Delta$ are two symmetric generating sets for $G$ so that $\Sigma\subseteq (F\Delta)^n$ and $\Delta\subseteq (E\Sigma)^m$ for some finite sets $F,E\subseteq G$ and $n,m\geqslant 1$, then the two word metrics $\rho_\Sigma$ and $\rho_\Delta$ are quasi-isometric. Indeed, it suffices to notice that, in this case, $\Sigma$ is $\rho_\Delta$-bounded and vice versa.
This, in particular, applies when $\Sigma$ and $\Delta$ both have property (OB) relative to $G$. 

\begin{defi}
A metric space $(X,d)$ is said to be
 {\em large scale geodesic} if there is $K\geqslant 1$ so that, for all $x,y\in X$, there are $z_0=x,z_1,z_2,\ldots,z_n=y$ so that $d(z_i,z_{i+1})\leqslant K$ and 
$$
\sum_{i=0}^{n-1}d(z_i,z_{i+1})\leqslant K\cdot d(x,y).
$$
\end{defi}
For example, if $\mathbb X$ is a connected graph, then the shortest path metric $\rho$ on $\mathbb X$ makes $(\mathbb X, \rho)$ large scale geodesic with constant $K=1$.

We should note two well-known facts about large scale geodecity that can easily be checked by hand.
\begin{lemme}\label{facts geodecity}
\begin{enumerate}
\item Large scale geodecity is a quasi-isometric invariant of metric spaces.
\item If $F\colon (X,d_X)\til (Y,d_Y)$  is a bornologous map from a large scale geodesic space $(X,d_X)$ to a metric space $(Y,d_Y)$, then $F$ is Lipschitz for large distances.
\end{enumerate}
\end{lemme}

In our current setup, we can now characterise the  maximal metrics among the metrically proper ones.

\begin{prop}\label{char maximal metric}
The following conditions are equivalent for  a metrically proper compatible left-invariant metric $d$ on a metrisable group $G$,
\begin{enumerate}
\item $d$ is maximal,
\item $(G,d)$ is large scale geodesic,
\item there is $\eps>0$ so that $B_\eps=\{g\in G\del d(g,1)\leqslant \eps\}$ generates $G$ and $d$ is quasi-isometric to the word metric $\rho_{B_\eps}$.
\end{enumerate}
\end{prop}

\begin{proof}(3)$\saa$(2): We note that the word metric $\rho_{B_\eps}$ is simply the shortest path metric on the Cayley graph of $(G,B_\eps)$, i.e., the graph whose vertex set is $G$ and whose  edges are $\{g,gs\}$, for $g\in G$ and $s\in B_\eps$. Thus, $(G,\rho_{B_\eps})$ is large scale geodesic and, since $d$ is quasi-isometric to $\rho_{B_\eps}$, so is  $(G,d)$.

(2)$\saa$(1): Assume that $(G,d)$ is large scale geodesic with constant $K\geqslant 1$. Suppose $\partial$ is any other compatible left-invariant metric on $G$. Since $d$ is metrically proper, the identity map from $(G,d)$ to $(G,\partial)$ is bornologous. By Lemma \ref{facts geodecity}, it follows that it is also Lipschitz for large distances, showing the maximality of $d$.

(1)$\saa$(3): Suppose $d$ is maximal. We claim that $G$ is generated by some closed ball $B_k=\{g\in G\del d(g,1)\leqslant k\}$. Note that, if this fails, then $G$ is the increasing union of the chain of proper open subgroups $V_n=\langle B_n\rangle$, $n\geqslant 1$. However, it is now easy, using Lemma \ref{birkhoff-kakutani}, to construct a new metric from the $H_k$ contradicting the maximality of $d$. First, complementing with a neighbourhood basis $V_0\supseteq V_{-1}\supseteq V_{-2}\supseteq \ldots \ni 1$ of symmetric open sets so that $V_{-n}^3\subseteq V_{-n+1}$, and letting $\partial$ denote the metric obtained via Lemma \ref{birkhoff-kakutani} from $(V_n)_{n\in \Z}$, we see that, for all $g\in B_n\setminus V_{n-1}\subseteq  V_n\setminus V_{n-1}$, we have
$$
\partial(g,1)\geqslant 2^{n-1}\geqslant n\geqslant d(g,1).
$$
Since $B_n\setminus V_{n-1}\neq \tom$ for infinitely many $n\geqslant 1$, this contradicts the maximality of $d$ and therefore  $G=V_k=\langle B_k\rangle$ for some $k\geqslant 1$. 

Let now $\partial$ denote the metric obtained from $B_k$ and $d$ via Lemma \ref{construction maximal metrics}. Then $d\leqslant \partial$ and, since $d$ is maximal, we have $\partial\leqslant K\cdot d+C$ for some constants $K,C$, showing that $d$ and $\partial$ are quasi-isometric. 
\end{proof}
Thus, the maximal metrics on a metrisable group $G$ are simply the compatible left-invariant metrics in the quasi-isometry class of word metrics $\rho_\Sigma$, where $\Sigma$ is any symmetric open generating set with property (OB) relative to $G$.

\begin{thm}
The following are equivalent for a separable metrisable group $G$, 
\begin{enumerate}
\item $G$ admits  a maximal compatible left-invariant metric $d$ ,
\item $G$ is generated by an open set with property (OB) relative to $G$,
\item $G$ has the local property (OB) and is not the union of a chain of proper open subgroups.
\end{enumerate}
\end{thm}

\begin{proof}
(3)$\saa$(2): Assume (3) and let $U\ni 1$ be an open set with property (OB) relative to $G$. Since $G$ is separable and is not the union of a chain of proper open subgroups, there must be a finite set $1\in F\subseteq G$ so that $G$ is generated by $UF$, which also has property (OB) relative to $G$. 

(2)$\saa$(1): Let $U$ be any open generating set with property (OB) relative to $G$ and put $V=(U\cup \{1\})(U\cup\{1\})\inv$. Then $V$ has property (OB) relative to $G$ and so $G$ has the local property (OB). We can therefore choose some metrically proper compatible left-invariant metric $d$ on $G$ and let $\eps>0$ be so that 
$$
V\subseteq B_\eps=\{g\in G\del d(g,1)\leqslant \eps\}.
$$
Let now $\partial$ denote the compatible left-invariant metric on $G$ given by Lemma \ref{construction maximal metrics}. By Proposition \ref{char maximal metric}, $\partial$ is maximal.

(1)$\saa$(3): If $d$ is maximal, it is metrically proper and thus $G$ has the local property (OB). Moreover, as in the proof of (1)$\saa$(3) in Proposition \ref{char maximal metric}, $G$ cannot be the union of a countable chain of proper open subgroups.
\end{proof}

Thus far, we have been able to one the one hand characterise the maximal metrics and also characterise the groups admitting such metrics. However, oftentimes it will be useful to have other criteria that guarantee existence. In the context of finitely generated groups, the main such criteria is the Milnor--\v{S}varc lemma \cite{milnor, svarc} of which we will have a close analogue. For this, we will need the following definition.
\begin{defi}
An isometric group action $G\curvearrowright (X,d)$ on a metric space is said to be {\em cobounded} if, for all $x\in X$, 
$$
\sup_{y\in X}d(y,G\cdot x)<\infty,
$$
or, equivalently, there is a set $A\subseteq X$ of finite diameter so that $X=G\cdot A$.
\end{defi}

\begin{thm}[First Milnor--\v{S}varc lemma]\label{milnor connected}
Suppose $G$ is a metrisable group with a metrically proper cobounded continuous isometric action $G\curvearrowright (X,d)$ on a connected metric space. Then $G$ admits a maximal compatible left-invariant metric $\partial$.
\end{thm}

\begin{proof}Since the action is cobounded, there is an open set $A\subseteq X$ of finite diameter so that $G\cdot A=X$.
We let 
$$
S=\{g\in G\del g\cdot A\cap A\neq \tom\}
$$
and observe that $S$ is an open neighbourhood of $1$ in $G$. Note that, since the action is metrically proper and ${\rm diam}_d(S\cdot x)<\infty$ for all $x\in A$, there is no sequence $g_n\til \infty$ in $S$, or, equivalently, $S$ has property (OB) relative to $G$.

To see that $G$ admits a maximal compatible left-invariant metric $\partial$, it now suffices to verify that $G$ is generated by $S$. For this, observe that, if $g,f\in G$, then
\[\begin{split}
\Big(g\langle S\rangle \cdot A\Big)\cap \Big( f\langle S\rangle\cdot A\Big)\neq \tom
&\saa\Big( \langle S\rangle f\inv g \langle S\rangle\cdot A\Big)\cap A\neq \tom\\
&\saa\Big( \langle S\rangle f\inv g \langle S\rangle\Big)\cap S\neq \tom\\
&\saa  f\inv g\in  \langle S\rangle\\
&\saa g\langle S\rangle= f\langle S\rangle.
\end{split}\]
Thus, distinct left cosets $g\langle S\rangle$ and $f\langle S\rangle$ give rise to disjoint open subsets $g\langle S\rangle\cdot A$ and $f\langle S\rangle \cdot A$ of $X$. However, $X=\bigcup_{g\in G}g\langle S\rangle\cdot A$ and $X$ is connected, which implies that there can only be a single left coset of $\langle S\rangle$, i.e., $G=\langle S\rangle$.
\end{proof}

\begin{thm}[Second Milnor--\v{S}varc lemma]\label{milnor large scale geodesic}
Suppose $G$ is a metrisable group with a metrically proper cobounded continuous isometric action $G\curvearrowright (X,d)$ on a large scale geodesic metric space. Then $G$ admits a maximal compatible left-invariant metric $\partial$.
Moreover, for every $x_0\in X$, the map
$$
g\in G\mapsto gx_0\in X
$$
is a quasi-isometry between $(G,\partial)$ and $(X,d)$.
\end{thm}

\begin{proof}
Fix some $x_0\in X$ and a compatible left-invariant metric $D\leqslant 1$ on $G$. Then, by the continuity of the action, the following defines a compatible left-invariant metric $\partial$ on $G$
$$
\partial (g,f)=D(g,f)+d(gx_0,fx_0).
$$ 
Moreover, since the action is metrically proper, so is the metric $\partial$. Now, since $D\leqslant 1$ and the action is cobounded, we see that $g\in G\mapsto gx_0\in X$ is a quasi-isometry between $(G,\partial)$ and $(X,d)$. Since $(X,d)$ is large scale geodesic, so is $(G,\partial)$, whence $\partial$ is maximal by Proposition \ref{char maximal metric}. 
\end{proof}

\begin{exa}\label{banach2}
Consider again the additive group $(X,+)$ of a Banach space $(X,\norm\cdot)$. By Example \ref{banach1}, the norm-metric is metrically proper on $(X,+)$ and hence the shift-action of $(X,+)$ on $X$ itself is a metrically proper transitive isometric action. Since also the norm-metric on $X$ is geodesic, it follows from Theorem \ref{milnor large scale geodesic}, that the identity map is a quasi-isometry between the abelian topological group $(X,+)$ and the Banach space $(X,\norm\cdot)$. In other words, the quasi-isometry type of the group $(X,+)$ is none other than the quasi-isometry type of $(X,\norm\cdot)$ itself, showing that, for Banach spaces, our theory coincides with the usual large scale geometry. 
\end{exa}

\begin{exa}\label{banach3}
Assume  $(X,\norm\cdot)$ is a separable Banach space. Then, by the Mazur--Ulam theorem, every surjective isometry $A\colon X\til X$ is affine, meaning that there is a linear isometry $f\colon X\til X$ and a vector $x\in X$ so that $A(y)=f(y)+x$ for $y\in X$. 

Now, let ${\rm Aff}(X,\norm\cdot)$ be the group of affine isometries of $(X,\norm\cdot)$ equipped with the topology of pointwise convergence on $(X,\norm\cdot)$, i.e., $g_i\til g$ if and only if $\norm{g_i(x)- g(x)}\til 0$ for all $x\in X$, and let ${\rm Isom}(X,\norm\cdot)$ be the closed subgroup consisting of linear isometries. Here the induced topology on ${\rm Isom}(X,\norm\cdot)$ is simply the strong operator topology. 
Then ${\rm Aff}(X,\norm\cdot)$ is a metrisable topological group and may be written as a topological semi-direct product 
$$
{\rm Aff}(X,\norm\cdot)={\rm Isom}(X,\norm\cdot)\ltimes (X,+).
$$
That is, ${\rm Isom}(X,\norm\cdot)\ltimes (X,+)$ is the topological product space ${\rm Isom}(X,\norm\cdot)\times (X,+)$ equipped with the group operation $(f,x)*(g,y)=(fg,x+f(y))$. 

Suppose now that ${\rm Isom}(X,\norm\cdot)$ has property (OB). We claim that ${\rm Aff}(X,\norm\cdot)$ is  quasi-isometric to $(X,\norm\cdot)$. To see this, note that ${\rm Isom}(X,\norm\cdot)$ and $(X,+)$ embed into ${\rm Isom}(X,\norm\cdot)\ltimes (X,+)$ via $f\mapsto (f,0)$ and $x\mapsto ({\rm Id}, x)$ respectively. Thus, as the  ball $B_\alpha=\{x\in X\del \norm{x}\leqslant \alpha\}$ has property (OB) relative to $(X,+)$ (cf. Example \ref{banach2}) and ${\rm Isom}(X,\norm\cdot)$ has property (OB),  the subsets $\{{\rm Id}\}\times B_\alpha$ and ${\rm Isom}(X,\norm\cdot)\times \{0\}$ have property (OB) relative to ${\rm Isom}(X,\norm\cdot)\ltimes (X,+)$. Therefore, also
$$
{\rm Isom}(X,\norm\cdot)\times B_\alpha= \Big(\{{\rm Id}\}\times B_\alpha\Big)*\Big({\rm Isom}(X,\norm\cdot)\times \{0\}\Big)
$$
has property (OB) relative to ${\rm Isom}(X,\norm\cdot)\ltimes (X,+)$ for all $\alpha$. Thus, if $(f_n, x_n)$ is a sequence so that $(f_n,x_n)\til \infty$, we see that $(f_n,x_n)$ must eventually leave each set ${\rm Isom}(X,\norm\cdot)\times B_\alpha$, which implies that $\norm{x_n}\til \infty$. In particular, this shows that the canonical isometric action 
$$
{\rm Isom}(X,\norm\cdot)\ltimes (X,+)\curvearrowright (X,\norm\cdot)
$$
given by $(f,x)(y)=f(y)+x$ is metrically proper. Since it is evidently transitive, we conclude by Theorem \ref{milnor large scale geodesic}, that the mapping $(f,x)\mapsto (f,x)(0)=x$ is a quasi-isometry between ${\rm Isom}(X,\norm\cdot)\ltimes (X,+)$ and $(X,\norm\cdot)$. I.e., ${\rm Aff}(X,\norm\cdot)$ is  quasi-isometric to $(X,\norm\cdot)$ as claimed.
\end{exa}

\begin{exa}
Since the unitary group $U(\ku H)$ of separable infinite-dimensional Hilbert space has property (OB) in the discrete topology \cite{ricard} and thus also in the strong operator topology, by Example \ref{banach3}, we see that the group of affine isometries of $\ku H$ is quasi-isometric to $\ku H$ itself.

Also,  S. Banach described the linear isometry groups of $\ell^p$, $1<p<\infty$, as consisting entirely of sign changes and permutations of the basis elements. Thus, the isometry group is the semi-direct product $S_\infty\ltimes \{-1,1\}^\N$ of two groups with property (OB). By Proposition 4.1 \cite{OB}, it follows that ${\rm Isom}(\ell^p)$ has property (OB) and thus, by Example \ref{banach2}, that the affine isometry group ${\rm Aff}(\ell^p)$ is  quasi-isometric to $\ell^p$.

By results due to C. W. Henson \cite{henson},  the $L^p$-lattice $L^p([0,1],\lambda)$, with $\lambda$ being Lebesgue measure and $1<p<\infty$, is $\omega$-categorical in  the sense of model theory for metric structures. This also implies that the Banach space reduct  $L^p([0,1],\lambda)$ is $\omega$-categorical and hence that the action by its isometry group  on the unit ball  is approximately oligomorphic. By Theorem 5.2 \cite{OB}, it follows that the isometry group ${\rm Isom}(L^p)$ has property (OB) and thus, as before, that the affine isometry group ${\rm Aff}(L^p)$ is quasi-isometric to $L^p$.

Now, by results of W. B. Johnson, J. Lindenstrauss and G. Schechtman \cite{johnson} (see also Theorem 10.21 \cite{lindenstrauss}), any Banach space quasi-isometric to $\ell^p$ for $1<p<\infty$ is, in fact,  linearly isomorphic to $\ell^p$. Also, for $1<p<q<\infty$, the spaces $L^p$ and $L^q$ are not coarsely equivalent since they then would be quasi-isometric (being geodesic spaces) and, by taking ultrapowers, would be Lipschitz equivalent, contradicting Corollary 7.8 \cite{lindenstrauss}.

Thus, it follows that all of ${\rm Aff}(\ell^p)$ and ${\rm Aff}(L^p)$ for $1<p<\infty$ have distinct quasi-isometry types and, in particular, cannot be isomorphic as topological groups.
\end{exa}

\begin{exa}\label{semidirect}
Suppose $d$ is a compatible left-invariant metric on a topological group $G$ and $K\normal G$ is closed normal subgroup of finite diameter. Then the Hausdorff distance $d_H$ on the quotient space $G/K$, defined by 
$$
d_H(gK,fK)=\max\big\{\sup_{a\in gK}\inf_{b\in fK}d(a,b),  \sup_{b\in fK}\inf_{a\in gK}d(a,b)\big\},
$$
is a compatible left-invariant metric on the group $G/K$ and, moreover, satisfies
$$
d_H(gK,fK)=\inf_{k\in K}d(g,fk)=\inf_{k\in K}d(gk,f).
$$

Suppose now that $G$ is a metrisable group and $K\normal G$ is a closed normal subgroup with property (OB) relative to $G$ so that the quotient group $G/K$ admits a maximal compatible left-invariant metric $\partial$. Then we claim that the quotient map $g\in G\mapsto gK\in G/K$ is a quasi-isometry  between $G$ and $G/K$.

To see this, note first that the left-shift $G\curvearrowright (G/K, \partial)$ is a transitive continuous isometric action of $G$ on a large scale geodesic metric space. By Theorem \ref{milnor large scale geodesic}, it suffices to show that the action is metrically proper. So suppose $g_n\til \infty$ in $G$ and fix a compatible left-invariant metric on $G$ witnessing this, i.e., $d(g_n,1)\til \infty$. Since $K$ must have finite $d$-diameter, it follows that $d(g_nK,1K)\til \infty$ and so, since $\partial$ is maximal on $G/K$ and $d_H$ is a compatible metric, we have $\partial(g_nK, 1K)\til \infty$. I.e., the action is metrically proper and 
$g\in G\mapsto gK\in G/K$ is a quasi-isometry.

Conversely, suppose that $d$ is a maximal compatible left-invariant metric on $G$ and that $K\normal G$ is a closed normal subgroup with property (OB) relative to $G$. Then we claim that $d_H$ is a maximal compatible left-invariant metric on $G/K$.

For this, let $\partial$ be any other compatible left-invariant metric on $G/K$. Then $D(g,f)=d(g,f)+\partial(gK,fK)$ defines a compatible left-invariant metric on $G$, whereby, using maximality of $d$,  there is a constant $C$ so that
$$
\partial(gK,fK)\leqslant D(g,f)\leqslant C\cdot d(g,f)+C
$$
for all $g,f\in G$. Without loss of generality, we make take $C\geqslant {\rm diam}_d(K)$, whereby $d_H(gK,fK)=\inf_{k\in K}d(gk,f)\geqslant d(g,f)-C$ and hence
$$
\partial(gK,fK)\leqslant C \cdot d(g,f)+C\leqslant C\cdot d_H(gK,fK)+2C,
$$
for all $g,f\in G$. In other words, $d_H$ is maximal on $G/K$.
\end{exa}

\begin{exa}\label{cameron-vershik}
P. J. Cameron and A. M. Vershik \cite{vershik} have shown that that there is an invariant metric $d$ on the group $\Z$ for which  the metric space $(\Z,d)$ is isometric to the rational Urysohn metric space $\Q\U$. 
Since $d$ is two-sided invariant, the topology $\tau$ it induces on $\Z$ is necessarily a group topology, i.e., the group operations are continuous. Thus, $(\Z,\tau)$ is a metrisable topological group and we claim that $(\Z,\tau)$ has a well-defined quasi-isometry type, namely, $\U$ or, equivalently, $\Q\U$.

To see this, we first verify that $d$ is metrically proper on $(\Z, \tau)$. For this, note that, since $(\Z, \tau)$ is isometric to $\Q\U$, we have that, for all $n,m \in \Z$ and $\eps>0$, if $r=\lceil\frac{d(n,m)}\eps\rceil$, then there are $k_0=n, k_1, k_2, \ldots, k_r=m\in\Z$ so that $d(k_{i-1},k_i)\leqslant \eps$. 
Thus, as $r$ is a function only of $\eps$ and of the distance $d(n,m)$, we see that $d$ satisfies the criteria in Example \ref{geodesic proper} and hence is metrically proper on $(\Z,\tau)$. Also, as $\Q\U$ is large scale geodesic, so is $(\Z,d)$. It follows that the shift action of the topological group  $(\Z,\tau)$ on $(\Z,d)$ is a metrically proper transitive action on a large scale geodesic space. So, by Theorem \ref{milnor large scale geodesic}, the identity map is a quasi-isometry between the topological group $(\Z,\tau)$ and the metric space $(\Z,d)$. As the latter is quasi-isometric to $\Q\U$, so is $(\Z,\tau)$. 

By taking the completion of $(\Z,\tau)$, this also provides us with monothetic Polish groups quasi-isometric to the Urysohn space  $\U$.
\end{exa}


\section{Affine actions on Banach spaces}
By the Mazur--Ulam Theorem, every surjective isometry $A$ of a Banach space $X$ is {\em affine}, that is, there are a unique invertible linear isometry $T\colon X\til X$ and a vector $\eta\in X$ so that $A$ is given by $A(\xi)=T(\xi)+\eta$ for all $\xi \in X$. 
It follows that, if $\alpha\colon G\curvearrowright X$ is an isometric action of a group $G$ on a Banach space $X$, there is an isometric  linear representation $\pi\colon G\curvearrowright X$, called the {\em linear part} of $\alpha$, and a corresponding {\em cocycle} $b\colon G\til X$ so that
$$
\alpha(g)\xi=\pi(g)\xi +b(g)
$$
for all $g\in G$ and $\xi \in X$. The cocycle $b$ then satisfies the {\em cocycle equation}
$$
b(gf)=\pi(g)b(f)+b(g)
$$
for $g,f\in G$. In particular, for all $g,f\in G$, it follows that $\alpha(g)0=b(g)$ and $b(gg\inv )=b(1)=0$, whereby
\[\begin{split}
\norm{ b(f)-b(g)}&=\norm{b(f)+b(g g\inv )-b(g)}\\
&=\norm{b(f)+\pi(g)b(g\inv)}\\
&=\norm{\pi(g\inv )b(f)+b(g\inv )}\\
&=\norm{b(g\inv f)}\\
&=\norm{\alpha(g\inv f)0-0}.
\end{split}\]

Using this, one sees that, if $G$ is a metrisable group with a compatible  left-invariant metric $d$ and the action $\alpha\colon G\curvearrowright X$ is continuous, then the cocycle $b\colon (G,d)\til X$ is actually  uniformly continuous. Moreover, if the metric $d$ is metrically proper, then, as $\alpha(B)0$ is norm bounded for every $d$-bounded set $B\subseteq G$, $b\colon (G,d)\til X$ is also bornologous. 
Finally, as$ \norm{ b(f)-b(g)}=\norm{b(g\inv f)}$, the cocycle $b\colon (G,d)\til X$ is a coarse embedding if and only if it is metrically proper, which, as $\norm{b( f)}=\norm{\alpha(f)0-0}$, happens if and only if $\alpha$ is a metrically proper action.

\begin{obs}\label{cocycle}
Assume $G$ is a metrisable group with a compatible  metrically proper left-invariant metric $d$. Suppose  $\alpha\colon G\curvearrowright X$ is a  a continuous affine isometric action on a Banach space $X$ with associated cocycle $b$. Then $b\colon G\til X$ is a coarse embedding if and only if $\alpha$, or equivalently $b$, is metrically proper.
\end{obs}

We shall adopt the standard notation $Z^1(G,\pi)$ for the vector space of continuous cocycles $b\colon G\til X$ associated to $\pi$, i.e., the continuous mappings $b$ satisfying the above cocycle equation.

Our goal in the following is twofold. On the one hand, we aim to offer various converses of the above observation, namely, we aim to show that, if $G$ is a metrisable group with a compatible  metrically proper left-invariant metric admitting a coarse embedding into a Banach space with some geometric property, then $G$ admits a metrically proper affine isometric action on a Banach space with this same property. On the other hand, we also wish to construct such actions without necessarily beginning from a coarse embedding, but rather from some other data on the group.

Our study naturally splits into four different cases depending on the required geometric properties of the Banach space upon which $G$ acts, namely, general, reflexive, super-reflexive and Hilbert spaces.
 

\subsection{The Arens--Eells space}\label{arens-eells}
Suppose $X$ is a non-empty set. The space $\M(X)$ of {\em molecules} over $X$ is the vector space of finitely supported real-valued functions $m\colon X\til \R$ with mean zero, i.e., so that
$$
\sum_{x\in X}m(x)=0.
$$
By induction on the size of its support, we note that every molecule $m\in \M(X)$ can be written as a finite linear combination of {\em atoms},  i.e., the molecules of the form
$$
m_{x,y}=\delta_x-\delta_y,
$$
where $x,y\in X$ and $\delta_x$ is the Dirac measure at $x$.

If, moreover, $(X,d)$ is a metric space,  we can define the {\em Arens--Eells norm} on $\M(X)$, by the formula
$$
{\norm {m}}_{\AE}=\inf\Big( \sum_{i=1}^n |a_i| d(x_i,y_i) \;\Del\; m=\sum_{i=1}^na_im_{x_i,y_i}\Big).
$$
A simple application of the triangle inequality for $d$ shows that, in the computation of the norm, the infimum is attained at some presentation $m=\sum_{i=1}^na_im_{x_i,y_i}$ where $x_i$ and $y_i$ all belong to the support of $m$. This also shows that $\norm{\cdot}_{\AE}$ is strictly positive on non-zero molecules, which verifies that it is indeed a norm. 
Moreover, as is well-known (see, e.g., \cite{weaver}), the norm is equivalently computed by
$$
\norm m_{{\AE}}=\sup\big( \sum_{x\in X}m(x)f(x)\del f\colon X\til \R \text{ is $1$-Lipschitz }\big),
$$
and so, in particular, $\norm{m_{x,y}}_{\AE}=d(x,y)$ for all $x,y\in X$. 
The completion of $\M(X)$ with respect to the norm $\norm\cdot_{\AE}$ will be denoted by ${\AE}(X,d)$, which we call the {\em Arens-Eells} space of $(X,d)$. 
Since the set of molecules that are rational linear combinations of atoms with support in a given dense subset of $X$ is dense in ${\text{\AE}}(X,d)$, one immediately sees that ${\text{\AE}}(X,d)$ is a separable Banach space  provided $(X,d)$ itself is separable. See the book by N. Weaver \cite{weaver} for further information on $\AE(X,d)$.

Now, if $G$ is a group acting by isometries on $(X,d)$, one immediately obtains an isometric linear action  $\pi\colon G\curvearrowright \big(\M(X), \norm{\cdot}_{\AE}\big)$ via 
$$
\pi(g) m=m(g\inv\;\cdot\;),
$$
or, equivalently,
$$
\pi(g)\big( \sum_{i=1}^na_im_{x_i,y_i}\big)=\sum_{i=1}^na_im_{gx_i,gy_i}
$$ 
for any molecule $m=\sum_{i=1}^na_im_{x_i,y_i}\in \M(X)$ and $g\in G$. Again, this action, being isometric, extends automatically and uniquely to the completion $\AE(X,d)$ of $\M(X)$.

The Arens--Eells space can be seen as a linearisation of the metric space $(X,d)$ by the following procedure. Fix any base point $e\in X$ and let $\phi_e\colon X\til \M(X)$ be the injection defined by 
$$
\phi_e(x)=m_{x,e}.
$$
Since 
$$
\norm{\phi_e(x)-\phi_e(y)}_{\AE}=\norm{m_{x,e}-m_{y,e}}_{\AE}=\norm{m_{x,y}}_{\AE}=d(x,y)
$$ 
for all $x,y\in X$, the map $\phi_e$ is an isometric embedding of $(X,d)$ into the Banach space $\AE(X,d)$. 

Similarly, associated to the isometric linear representation $\pi\colon G\curvearrowright \AE(X,d)$, we can construct a cocycle $b_e\colon G\til \AE(X,d)$ via the formula
$$
b_e(g)=m_{ge,e}.
$$
To verify that $b_e\in Z^1(G,\pi)$, i.e., that $b_e$ satisfies the cocycle equation, note that for $g,h\in G$
$$
b_e(gh)=m_{ghe,e}=m_{ge,e}+m_{ghe,ge}=b_e(g)+\pi(g)\big(b_e(h)\big).
$$
We let $\alpha_e\colon G\curvearrowright \AE(X)$ denote the corresponding affine isometric representation with linear part $\pi$ and cocycle $b_e$. We remark that in this case the following diagram commutes for all $g\in G$.
$$
\begin{CD}
X @>g>> X\\
@V\phi_eVV   @VV\phi_eV\\
\M(X)@>>\alpha_e(g)> \M(X)
\end{CD}
$$
Namely, for any $x\in X$,
\[\begin{split}
\big(\alpha_e(g)\circ \phi_e\big)(x)
&=\alpha_e(g)(m_{x,e})\\
&=\pi(g)(m_{x,e})+b_e(g)\\
&=m_{gx,ge}+m_{ge,e}\\
&=m_{gx,e}\\
&=\big(\phi_e\circ g\big)(x).
\end{split}\]
Therefore, for any choice of base point $e\in X$, the map $\phi_e\colon (X,d)\til \AE(X,d)$ is an equivariant isometry between the isometric action $G\curvearrowright (X,d)$ and the affine isometric action $\alpha_e\colon G\curvearrowright \AE(X,d)$.

\begin{thm}\label{thm arens-eells}
Suppose $d$ is a compatible left-invariant metric on a topological group $G$. Then the affine isometric action $\alpha\colon G\curvearrowright \AE(G,d)$ with linear part $\pi\colon G\curvearrowright \AE(G,d)$ and cocycle $b\in Z^1(G,\pi)$ given by $b(g)=\delta_g-\delta_1$ is continuous and satisfies
$$
\norm{b(g)}_{\AE}=d(g,1)
$$
for all $g\in G$. 
In particular, if $d$ is metrically proper, then so is the action $\alpha$.
\end{thm}

\begin{proof}
The only point not addressed by the preceding discussion is the continuity of the action, which separates into continuity of $b\colon G\til \AE(G,d)$ and strong continuity of $\pi$. For $b$, note that
$$
\norm{b(g)-b(f)}_{\AE}=\norm{\delta_g-\delta_f}_{\AE}=d(g,f).
$$
For strong continuity of $\pi$, note that, since the linear span $\M(G)$ of the atoms $\delta_g-\delta_f$ is dense in $\AE(G,d)$, it suffices to verify that the map 
$$
h\in G\mapsto \pi(h)\big(\delta_g-\delta_f\big)\in \AE(G,d)
$$ 
is continuous for all $ g,f\in G$. But this follows again from
\[\begin{split}
\norm{\pi(h)\big(\delta_g-\delta_f\big)-\pi(k)\big(\delta_g-\delta_f\big)}_{\AE}
&\leqslant \norm{\delta_{hg}-\delta_{kg}}_{\AE}+\norm{\delta_{hf}-\delta_{kf}}_{\AE}\\
&=d(hg,kg)+d(hf,kf),
\end{split}\]
thus finishing the proof.
\end{proof}

\subsection{Kernels conditionally of negative type and Hilbert spaces}
As the Arens--Eells space in general has very bad geometric properties even starting from a fairly well-behaved metric spaces, we are interested in other constructions that preserve more of the initial metric properties of $(X,d)$.
The most regular case is of course Hilbert spaces, for which we need some background material on kernels conditionally of negative type. The well-known construction presented here originates in work of E. H. Moore \cite{moore}.
\begin{defi}
A {\em kernel conditionally of negative type} on a set $X$ is a  function $\Psi\colon X\times X\til \R$  so that
\begin{enumerate}
\item $\Psi(x,x)=0$ and $\Psi(x,y)=\Psi(y,x)$ for all $x,y\in X$,
\item for all $x_1,\ldots, x_n\in X$ and $r_1,\ldots, r_n\in \R$ with $\sum_{i=1}^nr_i=0$, we have
$$
\sum_{i=1}^n\sum_{j=1}^n r_ir_j\Psi(x_i,x_j)\leqslant 0.
$$
\end{enumerate}
\end{defi} 
For example, if $\sigma\colon X\til \ku H$ is any mapping from $X$ into a Hilbert space $\ku H$, then a simple calculation shows that
$$
\sum_{i=1}^n\sum_{j=1}^n r_ir_j\norm{\sigma(x_i)-\sigma(x_j)}^2=-2\Norm{\sum_{i=1}^nr_i\sigma(x_i)}\leqslant 0,
$$
whenever $\sum_{i=1}^nr_i=0$, which implies that $\Psi(x,y)=\norm{\sigma(x)-\sigma(y)}^2$ is a kernel conditionally of negative type.

Now, if $\Psi$ is a kernel conditionally of negative type on  a set $X$ and $\M(X)$, as before,  denotes the vector space of finitely supported real valued functions $\xi$ on $X$ of mean $0$, i.e., $\sum_{x\in X}\xi(x)=0$, we can define a positive symmetric linear form  $\langle\cdot\del \cdot\rangle_\Psi$ on $\M(X)$ by
$$
\Big\langle\sum_{i=1}^nr_i\delta_{x_i}\Del \sum_{j=1}^ks_j\delta_{y_i}\Big\rangle_\Psi=-\frac 12\sum_{i=1}^n\sum_{j=1}^k r_is_j\Psi(x_i,y_j).
$$
Also, if $N_\Psi$ denotes the null-space
$$
N_\Psi=\{\xi\in \M(X)\del \langle\xi\del \xi\rangle_\Psi=0\},
$$
then $\langle\cdot\del \cdot\rangle_\Psi$ defines an inner product on the quotient $\M(X)/N_\Psi$ and 
we obtain a real Hilbert space $\ku K$ as the completion of $\M(X)/N_\Psi$ with respect to $\langle\cdot\del \cdot\rangle_\Psi$.

We remark that, if $\Psi$ is defined by a map $\sigma\colon X\til \ku H$ as above and $e\in X$ is any choice of base point, the map $\phi_e\colon X\til \ku K$ defined by $\phi_e(x)=\delta_x-\delta_e$ satisfies $\norm{\phi_e(x)-\phi_e(y)}_\ku K=\norm{\sigma(x)-\sigma(y)}_\ku H$. Indeed,
\[\begin{split}
\norm{\phi_e(x)-\phi_e(y)}^2_\ku K
&=\langle \phi_e(x)-\phi_e(y)\del \phi_e(x)-\phi_e(y)\rangle\\
&=\langle \delta_x-\delta_y\del \delta_x-\delta_y\rangle\\
&=-\frac 12\big(\Psi(x,x)+\Psi(y,y)-\Psi(x,y)-\Psi(y,x)\big)\\
&=\Psi(x,y)\\
&=\norm{\sigma(x)-\sigma(y)}_\ku H^2.
\end{split}\]

Also, if $G\curvearrowright X$ is an action of a group $G$ on $X$ and $\Psi$ is $G$-invariant, i.e., $\Psi(gx,gy)=\Psi(x,y)$, this action lifts to an action  $\pi\colon G\curvearrowright\M(X)$ preserving the form $\langle\cdot\del\cdot\rangle_\Psi$  via $\pi(g)\xi=\xi(g\inv\,\cdot\,)$. It follows that $\pi$ factors through to an orthogonal (i.e., isometric linear) representation $G\curvearrowright\ku K$.

Theorem \ref{maurey} below is originally due to I. Aharoni, B. Maurey and B. S. Mityagin \cite{maurey} for the case of abelian groups and is also worked out by Y. de Cornulier, R. Tessera and A. Valette  for compactly generated  locally compact groups in \cite{tessera}. Since more care is needed when dealing with general amenable groups as opposed to locally compact, we include a full proof.

\begin{thm}\label{maurey}
Suppose $d$ is a compatible left-invariant metric on an amenable  topological group $G$ and $\sigma\colon (G,d)\til \ku H$ is a uniformly continuous and bornologous map into a Hilbert space $\ku H$ with compression and expansion moduli $\kappa_1$ and $\kappa_2$. 
Then there is a continuous affine isometric action $\alpha\colon G\curvearrowright \ku K$  on a real Hilbert space  with associated cocycle $b\colon G\til \ku K$ so that 
$$
\kappa_1\big(d(g,1)\big)\leqslant \norm{b(g)}\leqslant \kappa_2\big(d(g,1)\big),
$$
for all $g\in G$.
\end{thm}

\begin{proof} 
For fixed $g,h\in G$, we define a function $\phi_{g,h}\colon G\til \R$ via
$$
\phi_{g,h}(f)=\norm{\sigma(fg)-\sigma(fh)}^2.
$$
Note first that $\phi_{g,h}\in \ell^\infty(G)$. For $d(fg,fh)=d(g,h)$ for all $f\in G$ and so, since $\sigma$ is bornologous, we have 
$$
\norm{\phi_{g,h}}_\infty=\sup_{f\in G}|\phi_{g,h}(f)|=\sup_{f\in G}\norm{\sigma(fg)-\sigma(fh)}^2\leqslant  \kappa_2\big(d(g,h)\big)^2<\infty.
$$

Secondly, we claim that $\phi_{g,h}$ is {\em left}-uniformly continuous, i.e., that  for all $\eps>0$ there is $W\ni 1$ open so that $|\phi_{g,h}(f)-\phi_{g,h}(fw)|<\eps$, whenever $f\in G$ and $w\in W$. To see this, take some $\eta>0$ so that $4\eta\norm{\phi_{g,h}}_\infty+4\eta^2<\eps$ and find, by uniform continuity of $\sigma$, some open $V\ni1$ so that   $\norm{\sigma(f)-\sigma(fv)}<\eta$ for all $f\in G$ and $v\in V$. 
Pick also $W\ni 1$ open so that $Wg\subseteq gV$ and $Wh\subseteq hV$. Then, if $f\in G$ and $w\in W$, there are $v_1,v_2\in V$ so that $wg=gv_1$ and $wh=hv_2$, whence
\[\begin{split}
\big|\phi_{g,h}(f)-\phi_{g,h}(fw)\big|
&=\Big|  \norm{\sigma(fg)-\sigma(fh)}^2- \norm{\sigma(fwg)-\sigma(fwh)}^2\Big|  \\
&=\Big|  \norm{\sigma(fg)-\sigma(fh)}^2- \norm{\sigma(fgv_1)-\sigma(fhv_2)}^2\Big|\\ 
&<4\eta\norm{\phi_{g,h}}_\infty+4\eta^2\\
&<\eps.
\end{split}\]
Thus, every $\phi_{g,h}$ belongs to the closed linear subspace ${\rm LUC}(G)\subseteq \ell^\infty(G)$ of left-uniformly continuous bounded real-valued functions on $G$ and a similar calculation shows that the map $(g,h)\in G\times G\mapsto \phi_{g,h}\in \ell^\infty(G)$ is continuous. 

Now, since $G$ is amenable, there exists a mean $m$ on ${\rm LUC}(G)$ invariant under the {\em right}-regular representation $\rho\colon G\curvearrowright {\rm LUC}(G)$ given by $\rho(g)\big(\phi\big)=\phi(\,\cdot\, g)$. Using this, we can define a continuous kernel $\Psi\colon G\times G\til \R$ by
$$
\Psi(g,h)=m(\phi_{g,h})
$$
and  note that $\Psi(fg,fh)=m(\phi_{fg,fh})=m\big(\rho(f)\big(\phi_{g,h}\big)\big)=m(\phi_{g,h})=\Psi(g,h)$ for all $g,h,f\in G$. 

We claim that $\Psi$ is a kernel conditionally of negative type. To verify this, let $g_1,\ldots, g_n\in G$ and $r_1,\ldots, r_n\in \R$ with $\sum_{i=1}^nr_i=0$. Then, for all $f\in G$, 
$$
\sum_{i=1}^n\sum_{j=1}^n r_ir_j\phi_{g_i,g_j}(f)=\sum_{i=1}^n\sum_{j=1}^n r_ir_j\norm{\sigma(fg_i)-\sigma(fg_j)}^2\leqslant 0,
$$
since $(g,h)\mapsto \norm{\sigma(fg)-\sigma(fh)}^2$ is a kernel conditionally of negative type. Since $m$ is positive, it follows that also
$$
\sum_{i=1}^n\sum_{j=1}^n r_ir_j\Psi(g_i,g_j)
= m\Big(\sum_{i=1}^n\sum_{j=1}^n r_ir_j\phi_{g_i,g_j}\Big)
\leqslant 0.
$$

As above, we define a positive symmetric form $\langle\cdot\del \cdot\rangle_\Psi$ on $\M(G)$.
Note that, since $\Psi$ is left-invariant, the form $\langle\cdot\del \cdot\rangle_\Psi$ is invariant under the left-regular representation $\lambda\colon G\curvearrowright \M(G)$ given by $\lambda(g)(\xi)=\xi(g\inv\,\cdot\,)$ and so $\lambda$ induces a strongly continuous  orthogonal representation $\pi$ of $G$ on the Hilbert space completion $\ku K$ of $\M(G)/N_\Psi$.

Moreover, as is easily checked, the map $b\colon G\til \ku K$ given by $b(g)=(\delta_g-\delta_1)+N_\Psi$ is a cocyle for $\pi$. Now
$$
\norm{b(g)}^2=\langle \delta_g-\delta_1 \del \delta_g-\delta_1\rangle_\Psi=\Psi(g,1)=m(\phi_{g,1}),
$$
so, since 
$$
\kappa_1\big(d(g,1)\big)^2\leqslant \phi_{g,1}\leqslant \kappa_2\big(d(g,1)\big)^2,
$$
it follows from the positivity of $m$ that
$$
\kappa_1\big(d(g,1)\big)\leqslant \norm{b(g)}\leqslant \kappa_2\big(d(g,1)\big),
$$
which proves the theorem.
\end{proof}

We can now extend the definition of the Haagerup property from locally compact groups to the full category of metrisable groups.
\begin{defi}
A metrisable group $G$ is said to have the {\em Haagerup property} if it admits a metrically proper continuous affine isometric action on a Hilbert space.
\end{defi}

Thus, based on Theorem \ref{maurey}, we have the following reformulation of the Haagerup property for amenable metrisable groups.

\begin{thm}\label{haagerup equiv}
The following are equivalent for an amenable separable metrisable  group with the local property (OB),
\begin{enumerate}
\item $G$ admits a uniformly continuous coarse embedding into a Hilbert space $\eta \colon G\til \ku H$,
\item $G$ has the Haagerup property.
\end{enumerate}
\end{thm}

\begin{proof}(2)$\saa$(1): If $\alpha\colon G\curvearrowright \ku H$ is a metrically proper continuous affine isometric action, with corresponding cocycle $b\colon G\til \ku H$, then, by Observation \ref{cocycle}, $b\colon G\til \ku H$ is a uniformly continuous coarse embedding.

(1)$\saa$(2): Fix a metrically proper compatible left-invariant metric $d$ on $G$. By Theorem \ref{maurey}, there is a continuous affine isometric action $\alpha\colon G\curvearrowright \ku K$ on a Hilbert space $\ku K$ with associated cocycle $b\colon G\til \ku K$ so that, for all $g\in G$, 
$$
\kappa_1\big(d(g,1)\big)\leqslant \norm{b(g)}\leqslant\kappa_2\big(d(g,1)\big),
$$
where $\kappa_1(t)=\inf_{d(g,f)\geqslant t}\norm{\eta(g)-\eta(f)}$ and $\kappa_2(t)=\sup_{d(g,f)\leqslant t}\norm{\eta(g)-\eta(f)}$. Since $\eta$ is a coarse embedding, the cocycle $b\colon G\til \ku K$ is metrically proper and so is the action $\alpha$. 
\end{proof}

In the setting of amenable non-Archimedean Polish groups, the condition of uniform continuity of $\eta$ may be dropped. Indeed, fix an open subgroup $V\leqslant G$ with property (OB) relative to $G$ and note that, since $\eta$ is bornologous,  there is a constant $K>0$ so that $\norm{\eta(g)-\eta(f)}\leqslant K$ whenever $f\inv g\in V$.
Letting $X\subseteq G$ denote a set of left-coset representatives for $V$, we define $\sigma(g)=\eta(h)$, where $h\in X$ is the coset representative of $gV$. Then $\norm{\eta(g)-\sigma(g)}\leqslant K$ for all $g\in G$, so $\sigma\colon G\til \ku H$ is bornologous and clearly constant on left-cosets of $V$, whence also uniformly continuous with respect to some chosen metrically proper compatible left-invariant metric $d$ on $G$. So $\sigma$ is a uniformly continuous coarse embedding.

U. Haagerup \cite{haagerup} initially showed that finitely generated free groups have the Haagerup property. It is also known that amenable locally compact groups  \cite{BCV} (see also  \cite{CCJJV}) have the Haagerup property. However, this is not the case for amenable metrisable groups, as, e.g., the isometry group of the Urysohn metric space $\U$ provides a counter-example. We shall verify this in Section \ref{superrefl}.

There is also a converse to this. Namely, E. Guentner and J. Kaminker \cite{GK} showed that, if a finitely generated discrete group $G$ admits a affine isometric action on a Hilbert space whose cocycle $b$ growths faster than the square root of the word length, then $G$ is amenable (see \cite{tessera} for the generalisation to the locally compact case).


\subsection{Approximate compactness, super-reflexivity and Rademacher type}\label{superrefl}
Weakening the geometric restrictions on the phase space from euclidean to uniformly convex, we still have a result similar to Theorem \ref{maurey}. However, in this case, we must assume that the group in question is  approximately compact and not only amenable.

\begin{defi}
A topological group $G$ is said to be {\em approximately compact} if there is a countable chain $K_0\leqslant K_1\leqslant \ldots \leqslant G$ of compact subgroups whose union $\bigcup_nK_n$ is dense in $G$.
\end{defi}

This turns out to be a fairly common phenomenon among non-locally compact metrisable groups. For example, the unitary group $U(\ku H)$ of separable infinite-dimensional Hilbert space with the strong operator topology  is approximately compact. Indeed, if $\ku H_1\subseteq \ku H_2\subseteq \ldots\subseteq \ku H$ is an increasing exhaustive sequence of finite-dimensional subspaces and $U(n)$ denotes the group of unitaries pointwise fixing the orthogonal complement $\ku H_n^\perp$, then each $U(n)$ is compact and the union $\bigcup_nU(n)$ is dense in $U(\ku H)$.

More generally, as shown by P. de la Harpe \cite{harpe}, if $M$ is an approximately finite-dimensional von Neumann algebra, i.e., there is an increasing sequence $A_1\subseteq A_2\subseteq\ldots\subseteq M$ of finite-dimensional matrix algebras whose union is dense in $M$ with respect the strong operator topology, then the unitary subgroup $U(M)$ is approximately compact with respect to the strong operator topology. 

Similarly, if $G$ contains a locally finite dense subgroup, this will witness approximate compactness. Again this applies to, e.g., ${\rm Aut}([0,1],\lambda)$ with the weak topology, where the dyadic permutations are dense,  and ${\rm Isom}(\U)$ with the pointwise convergence topology (this even holds for the dense subgroup ${\rm Isom}(\Q\U)$ by a result of S. Solecki; see \cite{RZ} for a proof).

Of particular interest to us is the case of non-Archimedean Polish groups. By general techniques, these may be represented as automorphism groups of countable locally finite  (i.e., any finitely generated substructure is finite) ultrahomogeneous structures. And, in this setting, we have the following reformulation of approximate compactness.
\begin{prop}[A.S. Kechris \& C. Rosendal \cite{turbulence}]
Let $\bf M$ be a locally finite, countable, ultrahomogeneous structure.
Then ${\rm Aut}(\bf M)$ is approximately compact if and only if, for every finite substructure $\bf A\subseteq \bf M$ and all {\em partial} automorphisms $\phi_1,\ldots,
\phi_n$ of $\bf
A$, there is a larger finite substructure $\bf B$ with
$$
\bf A\subseteq \bf B\subseteq \bf M
$$
and {\em full} automorphisms $\psi_1,\ldots,\psi_n$ of $\bf B$ extending $\phi_1,\ldots,\phi_n$ respectively.
\end{prop}

We recall that a Banach space $V$ is {\em super-reflexive} if every other space crudely finitely representable in $V$ is reflexive. That is, if $X$ is a Banach space so that, for some fixed $K\geqslant 1$, every finite-dimensional subspace $F\subseteq X$ is $K$-isomorphic to a subspace of $V$, then $X$ is reflexive. In particular, every super-reflexive space is reflexive. Moreover, super-reflexive spaces are exactly those all of whose ultrapowers are reflexive. By a result of P. Enflo \cite{enflo} (see also G. Pisier \cite{pisier} for an improved result or \cite{fabian} for a general treatment), the super-reflexive spaces can also be characterised as those admitting an equivalent uniformly convex renorming.

For the case of super-reflexive spaces, we have the following result, which is due to V. Pestov \cite{pestov} in the case of a locally finite discrete group $G$. 
\begin{thm}\label{pestov}
Suppose $d$ is a compatible left-invariant metric on an approximately compact topological group $G$.  Assume that $\sigma\colon (G,d)\til E$ is a uniformly continuous and bornologous map into a super-reflexive Banach space $E$ with compression and expansion moduli $\kappa_1$ and $\kappa_2$.
Then there is a continuous affine isometric action of $G$ on a super-reflexive Banach space $V$ with corresponding cocycle $b$ so that 
$$
\kappa_1\big(d(f,1)\big)\leqslant \norm{b(f)}_V\leqslant \kappa_2\big(d(f,1)\big)
$$
for all $f\in G$.
\end{thm}

\begin{proof}By renorming $E$, we may suppose that $E$ is uniformly convex.
For a compact subgroup $K\leqslant G$, let $\mu$ denote the Haar measure on $K$ and $L^2(K, E)$ denote the Banach space of square integrable $E$-valued functions $\phi$ on $K$ with norm
$$
\norm{\phi}_{L^2}=\Big(\int_K\norm{\phi(g)}_E^2 \;d\mu(g)\Big)^{\frac 12}.
$$
Since $E$ uniformly convex, so is $L^2(K,E)$. Moreover,  for every other compact subgroup $C\leqslant G$, the probability spaces $K$ and $C$ are isomorphic,  whence the Banach spaces $L^2(K,E)$ and $L^2(C,E)$ are isometric. Thus, the modulus of uniform convexity of $L^2(K,E)$ is independent of the choice of $K$.
For every $f\in G$, we define an element $[f]_K\in L^2(K,E)$ by $[f]_K(g)=\sigma(gf)$ for all $g\in K$. Note that since $\sigma$ is continuous, so is $[f]_K$ and hence $[f]_K$ is automatically square integrable on the compact group $K$.

Let also $\rho\colon K\curvearrowright L^2(K,E)$ denote the right-regular representation given by $\big(\rho(g)\phi\big)(h)=\phi(hg)$ and note that, in particular, 
$$
\big(\rho(g)[f]_K\big)(h)=[f]_K(hg)=\sigma(hgf)=[gf]_K(h),
$$
i.e., $\rho(g)[f]_K=[gf]_K$ for all $g\in K$ and $f\in G$.

Now, fix a non-principal ultrafilter $\ku U$ on $\N$ and let $K_0\leqslant K_1\leqslant\ldots\leqslant G$ be a countable chain of compact subgroups whose union is dense in $G$. Consider the ultraproduct
$$
W=\prod_\ku UL^2(K_n,E).
$$
That is, $W$ is the quotient of $\big(\bigoplus_n L^2(K_n,E)\big)_\infty$ by the subspace
$$
N_\ku U=\{(\phi_n)\in \big(\bigoplus_n L^2(K_n,E)\big)_\infty\del \lim_\ku U\norm{\phi_n}_{L^2}=0\}.
$$
For $(\phi_n)\in \big(\bigoplus_n L^2(K_n,E)\big)_\infty$, we denote its image in $W$ by $(\phi_n)_\ku U$. 
Since the spaces $L^2(K_n,E)$ are all uniformly convex with the same modulus of uniform convexity, the ultraproduct remains uniformly convex and hence super-reflexive.

Note that, for all $f,h\in G$ and $n\in \N$, 
\[\begin{split}
\Norm{[f]_{K_n}-[h]_{K_n}}_{L^2}
&=\Big(\int_{K_n}\norm{[f]_{K_n}(g)-[h]_{K_n}(g)}_E^2 \;d\mu(g)\Big)^{\frac 12}\\
&=\Big(\int_{K_n}\norm{\sigma(gf)-\sigma(gh)}_E^2 \;d\mu(g)\Big)^{\frac 12}\\
&\leqslant \Big(\int_{K_n}\kappa_2\big(d(gf,gh)\big)^2 \;d\mu(g)\Big)^{\frac 12}\\
&= \Big(\int_{K_n}\kappa_2\big(d(f,h)\big)^2 \;d\mu(g)\Big)^{\frac 12}\\
&=\kappa_2\big(d(f,h)\big),
\end{split}\]
so the sequence $\big([f]_{K_n}-[h]_{K_n}\big)_n$ is uniformly bounded in the $L^2$-norms and thus belongs to $\big(\bigoplus_n L^2(K_n,E)\big)_\infty$.
By the same reasoning, we also note that  $\Norm{[f]_{K_n}-[h]_{K_n}}_{L^2}\geqslant \kappa_1\big(d(f,h)\big)$. It follows that, for all $f,h\in G$, 
\begin{equation}\label{eq super}
\kappa_1\big(d(f,h)\big)\leqslant \Norm{\big([f]_{K_n}-[h]_{K_n}\big)_\ku U}_{W}\leqslant \kappa_2\big(d(f,h)\big).
\end{equation}

By the above, it follows that we can define a map $b\colon G\til W$ by setting $b(f)=([f]_{K_n}-[1]_{K_n})_\ku U$.  Note that, since $\sigma$ is uniformly continuous, for all $\eps>0$ there is $\delta>0$ so that $\kappa_2(\delta)<\eps$. In particular, 
$$
\Norm{b(f)-b(h)}_W=\Norm{\big([f]_{K_n}-[h]_{K_n}\big)_\ku U}_{W}\leqslant \kappa_2\big(d(f,h)\big)<\eps,
$$
whenever $d(f,h)<\delta$. Thus, $b$ is uniformly continuous.

Also, for $g\in \bigcup_nK_n$, the right-regular representation $\rho(g)$ defines a linear isometry of $L^2(K_n,E)$ for all but finitely many $n\in \N$. Therefore, as the ultrafilter $\ku U$ is non-principal, this means that we can define an isometric linear representation
$$
\tilde\rho \colon \bigcup_nK_n\curvearrowright W
$$
by letting 
$$
\tilde\rho(g)\big((\phi_n)_\ku U\big)=\big(\rho(f)\phi_n\big)_\ku U.
$$

We claim that $b\in Z^1(\bigcup_nK_n,\tilde \rho)$, i.e., that $b$ satisfies the cocycle identity $b(fh)=\tilde\rho(f)b(h)+b(f)$ for $f,h\in \bigcup_nK_n$. To see this, note that
\[\begin{split}
\tilde\rho(f)b(h)+b(f)
&=\tilde\rho(f)\big([h]_{K_n}-[1]_{K_n}\big)_\ku U+\big([f]_{K_n}-[1]_{K_n}\big)_\ku U\\
&= \big([fh]_{K_n}-[f]_{K_n}\big)_\ku U+\big([f]_{K_n}-[1]_{K_n}\big)_\ku U\\
&= \big([fh]_{K_n}-[1]_{K_n}\big)_\ku U\\
&=b(fh).
\end{split}\]
In particular, we see that the linear span of $b[\bigcup_nK_n]$ is $\tilde\rho[\bigcup_nK_n]$-invariant. Moreover, since every $\tilde\rho(f)$ is an isometry and $b$ is continuous, the same holds for the closed linear span $V\subseteq W$ of $b[G]$. 

We claim that, for every $\xi\in V$, the map $f\in \bigcup_nK_n\mapsto \tilde\rho(f)\xi$ is uniformly continuous. Since linear span of $b[\bigcup_nK_n]$ is dense in $V$, it suffices to prove this for $\xi\in b[\bigcup_nK_n]$. So fix some $g\in \bigcup_nK_n$ and note that, for $f,h\in \bigcup_nK_n$, we have
\[\begin{split}
\Norm{\tilde\rho(f)b(g)-\tilde\rho(h)b(g)}_W
&=\Norm{\big([fg]_{K_n}-[f]_{K_n}\big)_\ku U-\big([hg]_{K_n}-[h]_{K_n}\big)_\ku U}_W\\
&\leqslant\Norm{\big([fg]_{K_n}-[hg]_{K_n}\big)_\ku U}_W+\Norm{\big([h]_{K_n}-[f]_{K_n}\big)_\ku U}_W\\
&\leqslant\kappa_2\big( d(fg,hg)  \big)+\kappa_2\big( d(h,f)   \big).
\end{split}\]
Now, given $\eps>0$, there is $\delta>0$ so that $\kappa_2(\delta)<\frac \eps2$ and an $\eta>0$ so that $d(f,h)<\eta$ implies that $d(fg,hg)<\delta$. It follows that, provided $d(f,h)<\min \{\eta,\delta\}$, we have $\Norm{\tilde\rho(f)b(g)-\tilde\rho(h)b(g)}_W<\eps$, hence verifying uniform continuity.

Using our claim and the density of $\bigcup_nK_n$ in $G$, we can now uniquely extend $\tilde\rho$ to a strongly continuous isometric linear presentation $\tilde\rho\colon G\curvearrowright V$ so that the cocycle identity $b(fh)=\tilde\rho(f)b(h)+b(f)$ holds for all $f,h\in G$. 
It follows that we can define a continuous affine isometric action $\alpha\colon G\curvearrowright V $ by setting $\alpha(f)\xi=\tilde\rho(f)\xi+b(f)$ for $\xi\in V$ and $f\in G$. 

Finally, note that 
$$
\kappa_1\big(d(f,1)\big)\leqslant \norm{b(f)}_V\leqslant \kappa_2\big(d(f,1)\big)
$$ 
for all $f\in G$. 
\end{proof}

Repeating the proof of Theorem \ref{haagerup equiv} and noting that approximately compact metrisable groups are separable, we obtain the following equivalence.
\begin{thm}\label{super reflexive equiv}
The following are equivalent for an approximately compact metrisable group with the local property (OB),
\begin{enumerate}
\item $G$ admits a uniformly continuous  coarse embedding into a super-reflexive Banach  space $\eta \colon G\til E$,
\item $G$ admits a metrically proper continuous affine isometric action on a super-reflexive Banach space.
\end{enumerate}
\end{thm}

The proof of Theorem \ref{pestov} is fairly flexible and allows for several variations preserving different local structure of the Banach space $E$. As a concrete example, we shall study the preservation of Rademacher type in the above construction. For that we fix a {\em Rademacher sequence}, i.e., a sequence $(\eps_n)_{n=1}^\infty$ of mutually independent random variables $\eps_n\colon \Omega\til \{-1,1\}$, where $(\Omega, \mathbb P)$ is some probability space, so that $\mathbb P(\eps_n=-1)=\mathbb P(\eps_n=1)=\frac 12$. E.g., we could take $\Omega=\{-1,1\}^\N$ with the usual coin tossing measure and let $\eps_n(\omega)=\omega(n)$.
\begin{defi}
A Banach space $X$ is  said to have {\em type} $p$ for some $1\leqslant p\leqslant 2$ if there is a constant $C$ so that
$$
\Big(\mathbb E\NORM{\sum_{i=1}^n\eps_ix_i}^p\Big)^\frac 1p\leqslant C\cdot\Big(\sum_{i=1}^n\norm{x_i}^p\Big)^\frac 1p
$$
for every finite sequence $x_1,\ldots, x_n\in X$.

Similarly, $X$ has {\em cotype} $q$ for some $2\leqslant q<\infty$ if there is a constant $K$ so that 
$$
\Big(\sum_{i=1}^n\norm{x_i}^q\Big)^\frac 1q\leqslant K\cdot  \Big(\mathbb E\NORM{\sum_{i=1}^n\eps_ix_i}^q\Big)^\frac 1q
$$
for every finite sequence $x_1,\ldots, x_n\in X$.
\end{defi}
We note that, by the triangle inequality, every Banach space has type $1$. Similarly, by stipulation, every Banach space is said to have cotype $q=\infty$.

Whereas the $p$ in the formula $\Big(\sum_{i=1}^n\norm{x_i}^p\Big)^\frac 1p$ is essential, this is not so with the $p$ in $\Big(\mathbb E\NORM{\sum_{i=1}^n\eps_ix_i}^p\Big)^\frac 1p$. Indeed, the Kahane--Khintchine inequality (see \cite{albiac}) states that, for all $1<p<\infty$, there is a constant $C_p$ so that, for every Banach space $X$ and $x_1,\ldots, x_n\in X$, we have
$$
\mathbb E\NORM{\sum_{i=1}^n\eps_ix_i}
\leqslant 
\Big(\mathbb E\NORM{\sum_{i=1}^n\eps_ix_i}^p\Big)^\frac 1p
\leqslant 
C_p\cdot\mathbb E\NORM{\sum_{i=1}^n\eps_ix_i}.
$$
In particular, for any $p,q\in [1,\infty[$, the two expressions $\Big(\mathbb E\NORM{\sum_{i=1}^n\eps_ix_i}^p\Big)^\frac 1p$ and $\Big(\mathbb E\NORM{\sum_{i=1}^n\eps_ix_i}^q\Big)^\frac 1q$ differ at most by some fixed multiplicative constant independent of the space $X$ and the vectors $x_i\in X$. We refer the reader to \cite{albiac} for more information on Rademacher type and cotype.

\begin{thm}\label{type}
Suppose $G$ is an approximately compact separable metrisable group with the local property (OB) admitting a uniformly continuous coarse embedding into a Banach space $E$ with type $p$ and cotype $q$ for some $1\leqslant p\leqslant 2\leqslant q\leqslant \infty$. Then $G$ admits a metrically proper affine isometric action on a  space with type $p$ and cotype $q$.\end{thm}

\begin{proof}Suppose $\sigma\colon G\til E$ is a uniformly continuous coarse embedding into a space $E$ with type $p$ and cotype $q$. We then repeat the proof of Thorem \ref{pestov} using the fact that, as $E$ has type $p$ and cotype $q$, so do the spaces $L^2(K_n,E)$. Since also $L^2(K_n,E)$ and $L^2(K_m,E)$ are isometric for all $n,m$, the ultraproduct $W$ of the proof is isometric to the ultrapower of a single space $L^2(K_1,E)$. However, the ultrapower $X^\ku U$ of a Banach space $X$ is {\em finitely representable} in $X$, meaning that every finite-dimensional subspace $F\subseteq X^\ku U$ almost isometrically embeds into $X$. Thus, in particular, $W$ is finitely representable in $L^2(K_1,E)$ and therefore has type $p$ and cotype $q$. This similarly holds for the  subspace $V=\overline{\rm span}\big(b[G]\big)\subseteq W$, which finishes  the proof. 
\end{proof}

We mention that, by results of W. Orlicz and G. Nordlander (see \cite{albiac}),  the space $L^p$ has type $p$ and cotype $2$, whenever $1\leqslant p\leqslant 2$, and type $2$ and cotype $p$, whenever $2\leqslant p<\infty$. So Theorem \ref{type} applies, in particular, when $G$ is a separable approximately compact metrisable group with the local property (OB) admitting a uniformly continuous coarse embedding into an $L^p$ space.

In the interval $1\leqslant p\leqslant 2$,  Theorem \ref{maurey} gives us a somewhat better result, since it follows from results of J. Bretagnolle, D. Dacunha-Castelle and J.-L. Krivine \cite{bretagnolle} that $L^p$ coarsely embeds into $L^2$ for all $p\in [1,2]$.


\subsection{Stable metrics and reflexive spaces}
The case of reflexive spaces trivialises for locally compact second countable groups, since N. Brown and E. Guentner \cite{BG} showed that every countable discrete group admits a proper affine isometric action on a reflexive Banach space and U. Haagerup and  A. Przybyszewska \cite{haagerup-affine} generalised this to locally compact second countable groups.

However, for general metrisable groups, the situation is significantly more complicated. Indeed, results of A. Shtern \cite{shtern} and M. Megrelishvili \cite{megrelishvili2}, show that a topological group $G$ admits a topologically faithful isometric linear representation on a reflexive Banach space, i.e., $G$ is isomorphic to a subgroup of the linear isometry group of a reflexive Banach space with the strong operator topology, if and only if the continuous weakly almost periodic functions on $G$ separate points and closed sets. 
Moreover, there are examples, such as the group ${\rm Homeo}_+[0,1]$ of increasing homeomorphisms of the unit interval \cite{megrelishvili1}, that admit no non-trivial continuous linear actions on a reflexive space.

We recall that a bounded  function $\phi\colon G\til \R$ is said to be {\em weakly almost periodic} provided that its orbit $\lambda(G)\phi=\{\phi(g\inv\,\cdot\, )\del g\in G\}$ under the left regular representation is a relatively weakly compact subset of $\ell^\infty(G)$. Also, by a result of A. Grothendieck \cite{grothendieck}, the weakly almost periodic functions on $G$ may be characterised by the following double limit criterion.

\begin{thm}[A. Grothendieck \cite{grothendieck}]
A a bounded  function $\phi\colon G\til \R$ on a group $G$ is weakly almost periodic if and only if, for all sequences $(g_n)$ and $(f_m)$ in $G$ and ultrafilters $\ku U$ and $\ku V$ on $\N$, we have
$$
\lim_{n\til \ku U}\lim_{m\til \ku V}\phi(g_nf_m)=\lim_{m\til \ku V}\lim_{n\til \ku U}\phi(g_nf_m).
$$
\end{thm}
With a bit of effort, one may verify that, equivalently,  $\phi$ is weakly almost periodic if and only if, for all sequences $(g_n)$ and $(f_m)$ in $G$, we have
$$
\lim_{n\til \infty}\lim_{m\til \infty}\phi(g_nf_m)=\lim_{m\til \infty}\lim_{n\til \infty}\phi(g_nf_m).
$$
whenever the two limits exist.

Motivated by the notion of stability in model theory, J.-L. Krivine and B. Maurey \cite{KM} isolated the concept of a {\em stable norm} on a Banach space, which equivalently can be defined in terms of stability of the metric.

\begin{defi}
A metric $d$ on a set $X$ is said to be {\em stable} if, for all $d$-bounded sequences $(x_n)$ and $(y_m)$ in $X$ and ultrafilters $\ku U$ and $\ku V$ on $\N$, we have
$$
\lim_{n\til \ku U}\lim_{m\til \ku V}d(x_n,y_m)=\lim_{m\til \ku V}\lim_{n\til \ku U}d(x_n,y_m).
$$
\end{defi}
We remark that a simple, but tedious, inspection shows that, if $d$ is stable, then so is the uniformly equivalent bounded metric $D(x,y)=\max\{d(x,y),1\}$. 

Now, N. Kalton \cite{kalton} showed that every stable metric space may be coarsely embedded into a reflexive Banach space. Moreover, he also showed that, e.g., the Banach space $c_0$ does not admit a coarse embedding into a reflexive Banach space.

For the case of groups, I. Ben Yaacov, A. Berenstein and S. Ferri \cite{ben yaacov} showed that a metrisable group admits a compatible left-invariant and stable metric if and only if it admits a topologically faithful  isometric linear  representation on a reflexive space.

Our goal here is to provide a group theoretical counter-part of Kalton's theorem, that is, we wish to construct metrically proper continuous affine isometric actions on Banach spaces of topological groups admitting metrically proper stable compatible left-invariant metrics.

\begin{thm}\label{stable metric refl}
Suppose a topological group $G$ carries a compatible left-invariant metrically proper stable metric. Then $G$ admits a metrically proper continuous affine isometric action on a reflexive Banach space.
\end{thm}

We should mention that this theorem is far from establishing an equivalence. Indeed, the Tsirelson space $T$ \cite{tsirelson} is a separable reflexive Banach space not containing isomorphic copies of any $\ell^p$, $1\leqslant p<\infty$, nor of $c_0$. Now, as noted in Exampe \ref{banach2},  the norm-metric on the additive group $(X,+)$ of a Banach space $(X,\norm\cdot)$ is maximal.  It follows that the translation action of $T$ on itself is a metrically proper affine action on a reflexive space. However, by a result of Y. Raynaud (Thm. 4.1 \cite {raynaud}), if the additive group of an infinite-dimensional Banach space $E$ admits a compatible invariant stable metric, then $E$ must contain an isomorphic copy of some $\ell^p$, $1\leqslant p<\infty$. In other words, $T$ has no equivalent invariant stable metric, but has a metrically proper affine isometric action on a reflexive space.

In fact, the proof of Theorem \ref{stable metric refl} will also require something less than a stable metric, namely, the existence of a  sufficiently separating family of continuous weakly almost periodic functions.

\begin{thm}\label{wap refl}
Suppose $d$ is a compatible left-invariant metric on a topological group $G$ and assume  that, for all $\alpha>0$, there is a continuous weakly almost periodic function $\phi\in \ell^\infty(G)$  with $d$-bounded support so that $\phi\equiv 1$ on $D_\alpha=\{g\in G\del d(g,1)\leqslant \alpha\}$.

Then $G$ admits a continuous  isometric action $\pi\colon G\curvearrowright X$ on a reflexive Banach space $X$ with a corresponding continuous and metrically proper cocycle $b\colon (G,d)\til X$.
\end{thm}

\begin{proof}Under the given assumptions, we claim that, for every integer $n\geqslant 1$,  there is a  continuous weakly almost periodic function $0\leqslant \phi_n\leqslant 1$ on $G$ so that
\begin{enumerate}
\item $\norm{\phi_n}_\infty=\phi(1)=1$,
\item$\norm{\phi_n-\lambda(g)\phi_n}_\infty\leqslant \frac 1{4^n} \text{ for all } g\in D_n$ and 
\item ${\rm supp}(\phi_n)$ is $d$-bounded. 
\end{enumerate}

To see this, we pick inductively sequences of  continuous weakly almost periodic functions $(\psi_i)_{i=1}^{4^n}$ and radii $(r_i)_{i=0}^{4^n}$ so that
\begin{enumerate}
\item[(i)] $0=r_0<2n<r_1<r_1+2n<r_2<r_2+2n<r_3<\ldots<r_{4^n}$,
\item[(ii)] $0\leqslant \psi_i\leqslant 1$,
\item[(iii)] ${\psi}_{i}\equiv 1$ on $D_{r_{i-1}+n}$,
\item[(iv)] ${\rm supp}(\psi_i)\subseteq D_{r_i}$.
\end{enumerate}
Note first that, by the choice of $r_i$, the sequence
$$
D_{r_0+n}\setminus D_{r_0},\; D_{r_1}\setminus D_{r_0+n}, \;D_{r_1+n}\setminus D_{r_1},\; D_{r_2}\setminus D_{r_1+n}, \; \ldots\;, D_{r_{4^n}}\setminus D_{r_{4^n-1}+n}, \; G\setminus D_{r_{4^n}}
$$
partitions $G$. Also, for all $1\leqslant i\leqslant 4^n$, 
$$
\psi_1\equiv \ldots\equiv  \psi_i\equiv 0, \text{ while } \psi_{i+1}\equiv \ldots\equiv \psi_{4^n}\equiv 1\text{ on }D_{r_i+n}\setminus D_{r_i}
$$
and 
$$
\psi_1\equiv \ldots\equiv \psi_{i-1}\equiv 0, \text{ while } \psi_{i+1}\equiv \ldots\equiv  \psi_{4^n}\equiv 1\text{ on }D_{r_i}\setminus D_{r_{i-1}+n}.
$$
Setting  $\phi_n=\frac 1{4^n}\sum_{i=1}^{4^n}\psi_i$, we note that, for  all $1\leqslant i\leqslant 4^n$, 
$$
\phi_n\equiv \frac {4^n-i}{4^n} \quad \text{ on }D_{r_i+n}\setminus D_{r_i}
$$
and
$$
 \frac {4^n-i}{4^n}\leqslant \phi_n\leqslant  \frac {4^n-i+1}{4^n} \quad \text{ on }D_{r_{i}}\setminus D_{r_{i-1}+n}.
$$

Now, if $g\in D_n$ and $f\in G$, then $|d(g\inv f,1)-d(f,1)|=|d(f,g)-d(f,1)|\leqslant d(g,1)\leqslant n$. So, if $f$ belongs to some term in the above partition, then $g\inv f$ either belongs to the immediately preceding, the same or the immediately following term of the partition. By the above estimates on $\phi_n$, it follows that $|\phi_n(f)-\phi_n(g\inv f)|\leqslant \frac1{4^n}$.
In other words, for $g\in D_n$, we have
\[\begin{split}
\norm{\phi_n-\lambda(g)\phi_n}_\infty= \sup_{f\in G}|\phi_n(f)-\phi_n(g\inv f)|\leqslant \frac 1{4^n},
\end{split}\]
which verifies condition (2). Conditions (1) and (3) easily follow from the construction.

Consider now a specific $\phi_n$ as above and define 
$$
W_n=\ov{\rm conv}\big(\lambda(G)\phi_n\cup -\lambda(G)\phi_n\big)\subseteq \ell^\infty(G)
$$
and, for every $k\geqslant 1$, 
$$
U_{n,k}=2^kW_n+2^{-k}B_{\ell^\infty},
$$
where $B_{\ell^\infty}$ denotes the unit ball in ${\ell^\infty}(G)$. Let $\norm{\cdot}_{n,k}$ denote the gauge on ${\ell^\infty}(G)$ defined by $U_{n,k}$, i.e., 
$$
\norm{\psi}_{n,k}=\inf(\alpha>0\del \psi\in \alpha\cdot U_{n,k}).
$$

If $g\in D_n$, then $\norm{\phi_n-\lambda(g)\phi_n}_\infty\leqslant \frac 1{4^n}$ and so, for $k\leqslant n$, 
$$
\phi_n-\lambda(g)\phi_n\in \frac 1{2^n}\cdot 2^{-k}B_{\ell^\infty}\subseteq \frac 1{2^n}\cdot U_{n,k}.
$$
In particular, 
\begin{equation}\label{a}
\norm{\phi_n-\lambda(g)\phi_n}_{n,k}\leqslant \frac 1{2^n},\;\;\text{ for all } k\leqslant n \text{ and }g\in D_n.
\end{equation}

On the other hand,  for all $g\in G$ and $k$, we have $\phi_n-\lambda(g)\phi_n\in 2W_n\subseteq \frac1{2^{k-1}}U_{n,k}$. Therefore,
\begin{equation}\label{b}
\norm{\phi_n-\lambda(g)\phi_n}_{n,k}\leqslant \frac 1{2^{k-1} },\;\;\text{ for all } k\text{ and }g.
\end{equation}

Finally, if $g\notin ({\rm supp}\;\phi_n)\inv$, then  $\norm{\phi_n-\lambda(g)\phi_n}_{n,1}\geqslant \frac 12\norm{\phi_n-\lambda(g)\phi_n}_\infty\geqslant \frac 12$. So 
\begin{equation}\label{c}
\norm{\phi_n-\lambda(g)\phi_n}_{n,1}\geqslant \frac 12,\;\;\text{ for all } g\notin ({\rm supp}\;\phi_n)\inv.
\end{equation}

It follows from (\ref{a}) and (\ref{b}) that, for $g\in D_n$, we have 
\[\begin{split}
\sum_k\norm{\phi_n-\lambda(g)\phi_n}^2_{n,k}
&\leqslant \underbrace{\Big(\frac 1{2^n}\Big)^2+\ldots+\Big(\frac 1{2^n}\Big)^2}_{n\text{ times}}
+ \Big(\frac 1{2^{(n+1)-1}}\Big)^2+ \Big(\frac 1{2^{(n+2)-1}}\Big)^2+\ldots\\
&\leqslant\frac 1{2^{n-1}},
\end{split}\]
while using (\ref{c}) we have, for $g\notin ({\rm supp}\;\phi_n)\inv$, 
$$
\sum_k\norm{\phi_n-\lambda(g)\phi_n}^2_{n,k}\geqslant \frac1{4}.
$$

Define $\triple{\cdot}_n$ on $\ell^\infty(G)$ by $\triple{\psi}_n=\big(\sum_k\norm{\psi}_{n,k}^2\big)^{\frac 12}$ and set
$$
X_n=\{\psi\in \ov{\rm span}(\lambda(G)\phi_n)\subseteq \ell^\infty(G)\del \triple{\psi}_n<\infty\}\subseteq \ell^\infty(G).
$$
By the main result of  W. J. Davis, T. Figiel, W. B. Johnson and A. Pe\l czy\'nski \cite{dfjp}, the interpolation space $(X_n,\triple{\cdot}_n)$ is a  reflexive Banach space.  Moreover, since $W_n$ and $U_{n,k}$ are $\lambda(G)$-invariant subsets of $\ell^\infty(G)$, one sees that $\norm{\cdot}_{n,k}$ and $\triple{\cdot}_n$ are $\lambda(G)$-invariant and hence we have an isometric linear representation $\lambda\colon G\curvearrowright (X_n, \triple{\cdot}_n)$.

Note that, since $\phi_n\in W_n$, we have $\phi_n\in X_n$ and can therefore define a cocycle $b_n\colon G\til X_n$ associated to $\lambda$ by $b_n(g)=\phi_n-\lambda(g)\phi_n$. 
By the estimates above, we have
$$
\triple{b_n(g)}_n=\triple{\phi_n-\lambda(g)\phi_n}_n\leqslant\big( \frac 1{{\sqrt 2}}\big)^{n-1}
$$ 
for $g\in D_n$, while
$$
\triple{b_n(g)}_n=\triple{\phi_n-\lambda(g)\phi_n}_n\geqslant\frac1{2}
$$
for $g\notin ({\rm supp}\;\phi_n)\inv$.

Let now $Y=\big(\bigoplus_n(X_n,\triple{\cdot}_n)\big)_{\ell^2}$ denote the  $\ell^2$-sum of the spaces $(X_n,\triple{\cdot}_n)$. Let also $\pi\colon G\curvearrowright Y$ be the diagonal action and $b=\bigoplus b_n$ the corresponding cocycle. To see that $b$ is well-defined, note that, for $g\in D_n$, we have
\[\begin{split}
\norm{b(g)}_Y
&=\Big(\sum_{m=1}^\infty\triple{b_m(g)}_m^2\Big)^\frac 12\\
&=\Big(\text{finite}+\sum_{m=n}^\infty\triple{b_m(g)}_m^2\Big)^\frac 12\\
&\leqslant\Big(\text{finite}+\sum_{m=n}^\infty\frac 1{2^{m-1}}\Big)^\frac 12\\
&<\infty,
\end{split}\]
so $b(g)\in Y$.

Remark that, whenever $g\notin ({\rm supp}\;\phi_n)\inv$, we have
$$
\norm{b(g)}_Y\geqslant \Big(  \underbrace{\big(\frac 12\big)^2+\ldots+\big(\frac 12\big)^2}_{n\text{ times}}  \Big)^\frac 12=\frac {\sqrt n}2.
$$  
As $({\rm supp}\;\phi_n)\inv$ is $d$-bounded, this 
shows that the cocycle $b\colon (G,d)\til Y$ is metrically proper. We leave the verification that the action is continuous to the reader.
\end{proof}

Let us now see how to deduce Theorem \ref{stable metric refl} from Theorem \ref{wap refl}. So fix a compatible metrically proper stable left-invariant stable metric $d$ on $G$ with corresponding balls $D_\alpha$. Then, for every $\alpha>0$, we can define a continuous bounded weakly almost periodic function $\phi_\alpha\colon G\til \R$ by
$$
\phi_\alpha(g)=2-\min\Big\{1, \max\big\{\frac{d(g,1)}\alpha, 2\big\}\Big\}.
$$
We note that $\phi_\alpha$ has $d$-bounded support, while $\phi_\alpha\equiv 1$ on $D_\alpha$, thus verifying the conditions of Theorem \ref{wap refl}.

\begin{exa}
Since, by Example \ref{banach2},  the additive group $(X,+)$ of a Banach space $(X,\norm\cdot)$ is quasi-isometric to $(X,\norm\cdot)$ itself, these provide examples of metrisable groups admitting metrically proper affine isometric actions on Banach spaces with various types of geometry.

For example, $c_0$ does not admit a coarse embedding into a reflexive Banach space \cite{kalton}  and thus cannot have a metrically proper affine isometric action on a reflexive Banach space. 

Also, by results of M. Mendel and A. Naor \cite{naor}, $L^q$ does not embed coarsely into $L^p$, whenever $\max \{2,p\}<q<\infty$. Thus, $L^q$ cannot have a metrically proper affine isometric action on $L^p$ either. But, being super-reflexive, its shift-action on itself is a metrically proper affine isometric action on a super-reflexive space.
\end{exa}

\section{Open problems}

\begin{prob}
Suppose $M$ is a compact surface of genus $g\geqslant 1$. Does the group  of orientation preserving homeomorphisms, ${\rm Homeo}^+(M)$, have the local property (OB)? Does the identity component ${\rm Homeo}^+_0(M)$ have property (OB) relative to ${\rm Homeo}^+(M)$? We note that, if the latter holds, then ${\rm Homeo}^+(M)$ would be quasi-isometric to the mapping class group $\ku M(M)={\rm Homeo}^+(M)/{\rm Homeo}^+_0(M)$.
\end{prob}

\begin{prob}
Do diffeomorphism groups of smooth manifolds have well-defined quasi-isometry type? Cf. the results of \cite{brandenburgsky}, in which certain right-invariant metrics on groups of measure-preserving diffeomorphisms are shown to be unbounded.
\end{prob}

\begin{prob}
Is there an analogue of the Guentner--Kaminker result \cite{GK} valid for general metrisable groups? I.e., if a metrisable group $G$ admits a maximal metric and an affine isometric action on a Hilbert space with a cocycle growing faster that the square root of the distance, does it follow that $G$ is amenable? 
\end{prob}

\begin{prob}
Let $M$ be an approximately finite-dimensional von Neumann algebra, whence its unitary subgroup $U(M)$ is approximately compact \cite{harpe}. Does $U(M)$ have property (OB) or the local property (OB)? If so, on what kinds of Banach spaces does $U(M)$ act  metrically properly by affine isometric transformations?
\end{prob}

\begin{prob}
Find necessary and sufficient conditions for a metrisable group to have a metrically proper continuous affine isometric action on a reflexive Banach space. In particular, if $G$ admits such an action, what can be said about the weakly almost periodic functions on $G$?
\end{prob}

\begin{prob}
Does the isometry group ${\rm Isom}(\Q\U)$ of the rational Urysohn metric space have a continuous affine isometric action on a reflexive space without a fixed point?
\end{prob}

\begin{prob}
More generally, does a non-Archimedean Polish group have property (OB) if and only if all of its affine isometric actions on reflexive spaces fix a point? If $G$ has property (OB), the existence of such a fixed point follows immediately from the fixed point theorem of C. Ryll-Nardzewski \cite{ryll}.
\end{prob}

\begin{prob}
Suppose $\bf M$ is the countable atomic model of an $\omega$-stable theory $T$. Does ${\rm Aut}(\bf M)$ have the local property (OB)? We note that, if this is  the case, then ${\rm Aut}(\bf M)$ admits a metrically proper continuous affine isometric action on a reflexive Banach space.
\end{prob}



\end{document}